\pdfoutput=1
\documentclass[10pt,a4paper]{amsart}
\usepackage[english]{babel}
\usepackage{amssymb,xcolor}
\usepackage[pdfstartview=FitH,urlbordercolor=blue!50!cyan]{hyperref}
\usepackage{tikz}
\usetikzlibrary{matrix,arrows}

\usepackage{pdfsync} 

\numberwithin{equation}{section}

\theoremstyle{plain}
\newtheorem{theorem}{Theorem}[section]
\newtheorem{lemma}[theorem]{Lemma}
\newtheorem{proposition}[theorem]{Proposition}
\newtheorem{corollary}[theorem]{Corollary}
\theoremstyle{definition}
\newtheorem{definition}[theorem]{Definition}
\newtheorem{notations}[theorem]{Notations}
\newtheorem{remark}[theorem]{Remark}

\newcommand{\jc}{\ensuremath{\mathcal{J}}}
\newcommand{\ac}{\ensuremath{\mathcal{A}}}
\newcommand{\oc}{\ensuremath{\mathcal{O}}}
\newcommand{\fc}{\ensuremath{\mathcal{F}}}

\newcommand{\ec}{\ensuremath{\mathcal{E}}}
\newcommand{\gc}{\ensuremath{\mathcal{G}}}

\newcommand{\vc}{\ensuremath{\mathcal{V}}}
\newcommand{\lc}{\ensuremath{\mathcal{L}}}
\newcommand{\Sc}{\ensuremath{\mathcal{S}}}

\newcommand{\K}{\mathbb{K}}
\newcommand{\mL}{\mathbb{L}}
\newcommand{\C}{\mathbb{C}}
\newcommand{\Q}{\mathbb{Q}}
\newcommand{\A}{\mathbb{A}}
\newcommand{\As}{\mathbb{A}^*}
\newcommand{\Hom}{\mathrm{Hom}}
\newcommand{\End}{\mathrm{End}}
\renewcommand{\Im}{\mathrm{Im}\,}

\title[K-lattices and the Adelic Heisenberg group for CM curves]
{Some remarks on K-lattices and the Adelic Heisenberg group for CM curves}

\author[F.~D'Andrea]{Francesco D'Andrea}
\address[F.~D'Andrea]{Universit\`a di Napoli ``Federico II'', Dipartimento di Matematica e
Applicazioni ``R.~Caccioppoli'' 
and I.N.F.N.\ Sezione di Napoli, Complesso Monte S.~Angelo, Via Cinthia, 80126 Napoli.}
\email{francesco.dandrea@unina.it}

\author[D.~Franco]{Davide Franco}
\address[D.~Franco]{Universit\`a di Napoli ``Federico II'', Dipartimento di Matematica e
Applicazioni ``R.~Caccioppoli'', 
Complesso Monte S.~Angelo, Via Cinthia, 80126 Napoli.}
\email{davide.franco@unina.it}

\thanks{\rule{0pt}{9pt}\textit{Acknowledgments.} This research was partially supported by 
UniNA and Compagnia di San Paolo under the grant ``STAR Program 2013''.}

\begin{document}

\begin{abstract}
We define an adelic version of a CM elliptic curve $E$ which  is equipped with an action of the profinite completion of the endomorphism ring of $E$. 
The adelic elliptic curve so obtained is  provided   with a natural embedding into the adelic Heisenberg group. We embed into the adelic Heisenberg group the set of  commensurability classes of arithmetic $1$-dimensional $\mathbb{K}$-lattices (here and subsequently, $\mathbb{K}$ denotes a quadratic imaginary number field) and
define theta functions on it. We also embed the groupoid of commensurability modulo dilations into the union of adelic Heisenberg groups relative to a set of representatives of elliptic curves with $R$-multiplication ($R$ is   the ring of algebraic integers of  $\K$).
We thus get adelic theta functions   on the set of $1$-dimensional $\mathbb{K}$-lattices and on the groupoid of commensurability modulo dilations.
Adelic theta functions turn out to be acted by the adelic Heisenberg group and behave nicely under complex automorphisms (Theorems \ref{last} and \ref{last'}).
\end{abstract}

\maketitle

\thispagestyle{empty}

\section{Introduction}
The aim of this paper is to connect the natural action of the
Heisenberg group on adelic theta functions with the 
adelic action stemming from the main theorem of 
complex multiplication for elliptic curves. We   are also interested in defining an embedding of the moduli spaces of arithmetic $1$-dimensional $\mathbb{K}$-lattices (here and subsequently, $\K$ denotes a quadratic imaginary number field)
into the adelic Heisenberg group, in order to define on them theta functions with a nice behavior under complex automorphisms (Theorems \ref{last} and \ref{last'}).

After the seminal paper \cite{BC}, many efforts have been devoted in recent years to the construction of quantum  systems incorporating explicit class field theory   for an imaginary quadratic number field $\mathbb{K}$ (\cite{KMS}, \cite{LF}, \cite{HL}, \cite{Ch}, \cite{KMSII}). More  specifically, in \cite{KMS} it is exhibited
a quantum statistical mechanical system fully incorporating the explicit class field theory.
The main ingredients of \cite{KMS} was given  in terms of commensurability of $1$-dimensional $\mathbb{K}$-lattices. The connection between class field theory and quantum statistical mechanics is provided by a $C^*$-dynamical system containing an \textit{arithmetic subalgebra} $\mathcal{A}_{\mathbb{Q}}$ with symmetry group isomorphic to $Gal(\mathbb{K}^{\mathrm{ab}},\mathbb{K})$ and a set of \textit{fabulous states} sending $\mathcal{A}_{\mathbb{Q}}$ to $\mathbb{K}^{\mathrm{ab}}$. The symmetry group action turns out to be compatible with Galois' one. The arithmetic subalgebra is defined by means of the \textit{modular field}, namely by the field of modular function defined over $\mathbb{Q}^{\mathrm{ab}}$.

This work was intended as an  attempt to define canonical \textit{adelic theta functions} on commensurability classes of (arithmetic) $1$-dimensional $\mathbb{K}$-lattices
and on the groupoid of commensurability modulo dilations. The main ingredient of our construction is provided by  the \textit{adelic Heisenberg group} (\cite{Theta3}). Adelic theta functions have  a nice behavior under the Galois action which incorporates all the properties stated in the \textit{Main Theorem of Complex Multiplication} (see \S\S ~5 and 6).

Motivated by a purely algebraic definition of adelic theta function over an abelian variety and their deformations, David Mumford introduced, in a celebrated series of papers of the sixties (\cite{Mum}), the \textit{finite Heisenberg group} acting on the sections of an ample line bundle defined on the abelian variety. This led him to an adelic version of any abelian variety, defined as an extension of the set of its torsion points by the Barsotti-Tate group (\cite{Theta3}, Ch.~3). It turned out that  sections of the pull back of some line bundle on the \textit{tower of isogenies} (\cite{Theta3}, Definition 4.26), the so called \textit{adelic theta functions},  are acted on by an adelic version of the Heisenberg group (\cite{Theta3}, Ch.4).

In this work we apply Mumford's constructions to elliptic curves with complex multiplication (CM curves for short) and compare it with the moduli spaces of $1$-dimensional $\mathbb{K}$-lattices introduced in \cite{KMS}. We define an adelic version of a CM elliptic curve $E$ which  is equipped with an action of the profinite completion of the endomorphism ring of $E$. This is also provided   with a natural embedding into the adelic Heisenberg group. 
This allows us to incorporate the  endomorphism ring  of the CM curve into the definitions of Heisenberg group and theta functions. 
We thus get an interpretation by means of Class Field Theory (Theorems \ref{main}, \ref{main'}, \ref{last} and \ref{last'}) of the usual nice behavior of theta functions under automorphisms fixing the Hilbert field of $\mathbb{K}$ (\cite{Theta3}, Proposition 5.6). One of our main results is an extension of the \textit{Main Theorem of Complex Multiplication}  (\cite{Sil} II, Theorem 8.2) involving the Heisenberg group, which allows us to give a complete description of the behavior of theta functions under complex automorphisms (Theorem \ref{main}, Theorem \ref{main'}). 

Aimed at defining theta functions on them, we embed the set of commensurability classes of arithmetic K-lattices into the adelic Heisenberg group (Notations \ref{mapslattices}). We also embed the groupoid of commensurability modulo dilations into the union of finitely many adelic Heisenberg groups (indeed into the union of adelic Heisenberg groups
corresponding to a set of representatives of elliptic curves with $R$-multiplication,  modulo isomorphisms). We obtain theta functions on the set of $1$-dimensional $\mathbb{K}$-lattices and on the groupoid of commensurability modulo dilations which are equipped with an action of the Heisenberg group and exhibiting a nice behavior under complex automorphisms (Theorems \ref{last} and \ref{last'}, Notations \ref{lastnotations}).

The paper is organized as follows. In Section $2$  we collect some basic facts about \textit{ad{\`e}les} and \textit{adelic elliptic curves} that will be needed in the following.  In Section $3$ we recall the spaces of $1$-dimensional $\mathbb{K}$-lattices introduced in \cite{KMS} and compare them with adelic CM curves in order to obtain natural morphisms into the Heisenberg group. Furthermore, we  embed into the adelic Heisenberg group the set of  commensurability classes of arithmetic $1$-dimensional $\mathbb{K}$-lattices, and  we also   embed  the groupoid of commensurability modulo dilations into the union of adelic Heisenberg groups relative to a set of representatives of elliptic curves with $R$-multiplication (Notations \ref{mapslattices}).  In Section $4$ we introduce the Adelic Heisenberg group of a CM curve and embed the adelic CM curve into it by means of a symmetric line bundle. We also describe the  action  of the complex automorphisms fixing the Hilbert field of $\mathbb{K}$ on the Heisenberg group. In Section $5$ we state and prove a version of the \textit{Main Theorem of Complex Multiplication} (\cite{Sil} II, Theorem 8.2) concerning Heisenberg groups (Theorems \ref{sigmataucommute} and \ref{AMTCM}). In a nutshell, what such  theorems say   is that if two different CM curves are mapped into each other by a complex automorphism  then the embeddings into their Heisenberg groups can be made coherent with such a map. Finally,  in Section $6$  we introduce theta functions and study their behavior under complex automorphisms.

\section{Notations}

In this section we collect some basic facts about \textit{ad{\`e}les} and \textit{adelic elliptic curves},  mainly to fix our notations.

\subsection{Completions}

Here and subsequently,, $\K$  denotes a quadratic imaginary number field and $R$ denotes   the ring of algebraic integers of $\K$. We  denote   
  by $\mathcal{I}=\mathcal{I}_{R}$ the set of (integral) ideals of $R$, by $\mathcal{J}=\mathcal{J}_{R}$ the group of fractional  ideals of $R$, freely generated by the primes (\cite{Mar}, p.~91), and by $Cl(R)$ the \textit{ideal class group}   of $R$.
  
   As usual 
$$\hat{R}=\varprojlim_{I\in\mathcal{I}}\frac{R}{I}\subset \A$$ will be the completion of $R$, $\hat{R}^*$ the group of invertible elements of $\hat{R}$   and $\A = \A_{\K,f}$    the ring of finite ad{\`e}les of $\K$.
Recall that
$$I^{-1}= \{ x\in \K \mid xI\subset R \}\in \mathcal{J} $$
is a fractional ideal s.t.~$I \cdot I^{-1}=R$ . 
If $\Lambda \in \mathcal{J}$ is a fractional ideal of $\K$ then $\frac{I^{-1}\Lambda}{\Lambda}$ can be identified with a submodule of $\frac{\K}{\Lambda}$:

$$\frac{I^{-1}\Lambda}{\Lambda}=\left\{x\in \frac{\K}{\Lambda}\simeq \frac{\A}{\hat{\Lambda}} \;\bigg|\; xI=0 \right\}, \hskip2mm \hat{\Lambda}:=\Lambda\cdot \hat{R}\subset \A.$$
\begin{remark}
It is well known (see e.g.~\cite{Sil}, II Proposition 1.4) that $\frac{I^{-1}\Lambda}{\Lambda}$ is a free $\frac{R}{I}$-module of rank $1$. 
It is standard  to deduce from this fact  that, even though $\Lambda$ is a \textit{projective but usually non-free} $R$-module, the completion        
\begin{equation}
\label{lambdahatfree}  \hat{\Lambda}\simeq \varprojlim_{I\in\mathcal{I}}\frac{I^{-1}\Lambda}{\Lambda} 
\end{equation}
is a free $\hat{R}$-module of rank $1$.
We denote by $\hat{\Lambda}^*\subset \hat{\Lambda}$ the set of \textit{$\hat{R}$-module generators}  of $\hat{\Lambda}$ which    is obviously an \textit{$\hat{R}^*$-torsor}.   

 Essentially by the same reason as above, for any pair of fractional ideals $\Lambda, \Gamma\in \mathcal{J}$ one also gets
\begin{equation}
\label{homisomorphism}
\Hom_{R}\!\left(\frac{\K}{\Lambda},\frac{\K}{\Gamma}\right)= 
\varprojlim_{I\in \mathcal{I}}\Hom_{\frac{R}{I}}\!\left(\frac{I^{-1}\Lambda}{\Lambda},\frac{I^{-1}\Gamma}{\Gamma}\right)
\simeq \varprojlim_{I\in \mathcal{I}}\Hom_{\frac{R}{I}}\!\left(\frac{R}{I},\frac{R}{I}\right)
= \hat{R},
\end{equation}
where the  isomorphism can be made explicit through the choice of any pair of $\hat{R}$-module generators:  $\lambda \in \hat{\Lambda}^*$, $\gamma \in \hat{\Gamma}^*$.   
\end{remark}  

\subsection{Adelic elliptic curves with complex multiplication}
\label{AECCM}

As above let $\K$ be a quadratic imaginary number field and let $\Lambda \subset \K$ be a fractional ideal of $\K$. Then $E=E_{\Lambda}:=\frac{\mathbb{C}}{\Lambda}$ is an elliptic curve with $\End(E)=R$ (\cite{Sil}, p.~99). Following \cite{Sil} p.~102, for $a\in R$ we denote by $E[a]$ the $(a)$-\textit{torsion points} of $E$:
$$E[a]:= \{ x\in E \mid \hskip1mm ax=0\}\simeq \frac{(a)^{-1}\Lambda}{\Lambda}$$
(compare with \cite{Sil}, Proposition 1.4).
It is a standard fact that the  \textit{Barsotti-Tate module}
$$T(E):=\varprojlim E[a] $$  is a free $\hat{R}$-module of rank 1.
Similarly, we denote by $T(E)^*$ the $\hat{R}^*$-torsor   of $\hat{R}$-module generators  of $T(E)$.  
  
Now we imitate \cite{Theta3} Definition 4.1 to define an \textit{adelic version} of $E$ taking  into account its complex multiplication structure.
\begin{definition}
We define the \textit{adelic elliptic curve} associated to $E$ as
$$V(E):=\left\{(x_a)_{a\in R} \;\Big|\; \frac{a}{b}x_a=x_b, \text{ if } \, b\mid a, \; x_1\in E_{tor} \right\}.$$\par\noindent 
This is equipped with projections
$$\nu_b: V(E) \rightarrow E_{tor},\hskip3mm (x_a)_{a\in R}\rightarrow x_b.$$
By \cite{Theta3}, pp.~48-49, we have
\begin{equation}\label{tensor}
V(E)\simeq T(E)\otimes_{\mathbb{Z}}\Q
\end{equation}
and
\begin{equation}\label{exactsequence}
0\longrightarrow T(E) \longrightarrow V(E) \stackrel{\nu_1}{\longrightarrow}    E_{tor} \longrightarrow 0.   
\end{equation}
\end{definition}
\par\noindent
  
\begin{remark}\label{parametrization} ~
\begin{enumerate}\itemsep=3pt
\item If we fix an $\hat{R}$-module generator $u\in T(E)^*$ then \eqref{homisomorphism} implies that  any morphism $\phi \in \Hom_R\big(\frac{\K}{R},\frac{\K}{\Lambda}\big)$ can be represented by an element $\psi \in \hat{R}$. Then the multiplication by $\psi$ gives rise by \eqref{tensor} and \eqref{exactsequence} to     a commutative diagram:\vspace{3pt}
 \begin{center}
 \begin{tikzpicture}[description/.style={fill=white,inner sep=4pt},normal line/.style={->,font=\footnotesize,shorten >=2pt,shorten <=2pt}]
 \matrix (m) [matrix of math nodes, row sep=2.5em, column sep=3em, text height=1.5ex, text depth=0.25ex]
 {  0 & \hat{R} & \A & \frac{\mathbb{K}}{R} & 0 \\
    0 & T(E)  & V(E) & \frac{\mathbb{K}}{\Lambda} & 0 \\ };
 \path[normal line] (m-1-1) edge (m-1-2);
 \path[normal line] (m-1-2) edge (m-1-3);
 \path[normal line] (m-1-3) edge (m-1-4);
 \path[normal line] (m-1-4) edge (m-1-5);
 \path[normal line] (m-2-1) edge (m-2-2);
 \path[normal line] (m-2-2) edge (m-2-3);
 \path[normal line] (m-2-3) edge (m-2-4);
 \path[normal line] (m-2-4) edge (m-2-5);
 \path[normal line] (m-1-2) edge node[left] {$\psi$} (m-2-2);
 \path[normal line] (m-1-3) edge node[left] {$\psi$} (m-2-3);
 \path[normal line] (m-1-4) edge node[left] {$\phi$} (m-2-4);
 \end{tikzpicture}
 \end{center}
By the snake lemma  $\phi $ is an isomorphism if{}f $\psi$ belongs to $T(E)^*$.
\item
In particular, any choice of a $\hat{R}$-module generator $u\in  T(E)^*$  gives rise to a commutative diagram of $\hat{R}$-modules: \vspace{3pt}
 \begin{center}
 \begin{tikzpicture}[description/.style={fill=white,inner sep=4pt},normal line/.style={->,font=\footnotesize,shorten >=2pt,shorten <=2pt}]
 \matrix (m) [matrix of math nodes, row sep=2em, column sep=2.5em, text height=1.5ex, text depth=0.25ex]
 {  0 & \hat{R} & \A & \frac{\K}{R} & 0 \\
    0 & T(E) & V(E) & E_{tor} & 0 \\ };
 \path[normal line] (m-1-1) edge (m-1-2);
 \path[normal line] (m-1-2) edge (m-1-3);
 \path[normal line] (m-1-3) edge (m-1-4);
 \path[normal line] (m-1-4) edge (m-1-5);
 \path[normal line] (m-2-1) edge (m-2-2);
 \path[normal line] (m-2-2) edge (m-2-3);
 \path[normal line] (m-2-3) edge node[above] {$\nu_1$} (m-2-4);
 \path[normal line] (m-2-4) edge (m-2-5);
 \path[normal line,<->] (m-1-2) edge (m-2-2);
 \path[normal line,<->] (m-1-3) edge (m-2-3);
 \path[normal line,<->] (m-1-4) edge (m-2-4);
 \end{tikzpicture}
 \end{center}
\end{enumerate}
 \end{remark}

\section[The moduli space of K-lattices.]{The moduli space of \texorpdfstring{$\mathbb{K}$}{K}-lattices.}

Let us recall the following (see \cite{KMS}, Def.~4.1):
\begin{definition}~
\begin{enumerate}\itemsep=3pt
\item
A 1-dimensional $\K$ lattice $(\Lambda,\phi)$ is a finitely generated  $R$-submodule $\Lambda\subset \mathbb{C}$ s.t.~$\Lambda \otimes _R \K\simeq \K$, together with an $R$-morphism $\phi: \frac{\K}{R}\to \frac{\K\Lambda}{\Lambda}.$ 
\item
An invertible 1-dimensional $\K$ lattice $(\Lambda,\phi)$ is a finitely generated  $R$-submodule $\Lambda\subset \mathbb{C}$ s.t.~$\Lambda \otimes _R \K\simeq \K$, together with an $R$-isomorphism $\phi: \frac{\K}{R}\to \frac{\K\Lambda}{\Lambda}.$ 
\end{enumerate} 
\end{definition}\noindent
Since for any finitely generated  $R$-submodule $\Lambda\subset \mathbb{C}$ there exists $k\in \mathbb{C}$ s.t.~$k\Lambda\subset \mathbb{K}$ (\cite{CMBook}, Lemma 3.111),
we
also give the following:
\begin{definition} An arithmetic (invertible) 1-dimensional $\K$ lattice $(\Lambda,\phi)$ is a finitely generated  $R$-submodule $\Lambda\subset \mathbb{K}$ s.t.~$\Lambda \otimes _R \K\simeq \K$, together with a $R$-morphism (isomorphism) $\phi: \frac{\K}{R}\to \frac{\K}{\Lambda}.$ 
\end{definition}

The following result is an immediate consequence of Remark \ref{parametrization}.
\begin{theorem}
\label{sesm}
Any morphism $\phi \in \Hom_R\big(\frac{\K}{R},\frac{\K}{\Lambda}\big)$ gives rise to map of short exact sequences:
 \begin{center}
 \begin{tikzpicture}[description/.style={fill=white,inner sep=4pt},normal line/.style={->,font=\footnotesize,shorten >=4pt,shorten <=4pt}]
 \matrix (m) [matrix of math nodes, row sep=2.5em, column sep=2.5em, text height=1.5ex, text depth=0.25ex]
 {  0 & \hat{R} & \A & \frac{\mathbb{K}}{R} & 0 \\
    0 & T(E)  & V(E) & \frac{\mathbb{K}}{\Lambda} & 0 \\ };
 \path[normal line] (m-1-1) edge (m-1-2);
 \path[normal line] (m-1-2) edge (m-1-3);
 \path[normal line] (m-1-3) edge (m-1-4);
 \path[normal line] (m-1-4) edge (m-1-5);
 \path[normal line] (m-2-1) edge (m-2-2);
 \path[normal line] (m-2-2) edge (m-2-3);
 \path[normal line] (m-2-3) edge (m-2-4);
 \path[normal line] (m-2-4) edge (m-2-5);
 \path[normal line] (m-1-2) edge (m-2-2);
 \path[normal line] (m-1-3) edge (m-2-3);
 \path[normal line] (m-1-4) edge (m-2-4);
 \end{tikzpicture}
 \end{center}
\par\noindent
where $E:=\frac{\C}{\Lambda}$.
The corresponding $\K$-lattice is invertible if{}f the vertical maps are isomorphisms.
In particular any $\phi \in \Hom_R\big(\frac{\K}{R},\frac{\K}{\Lambda}\big)$ corresponds to a point of $x_{\phi}\in T(E)$ (the image of 
$1\in \hat{R}$) and the corresponding  $\K$-lattice is invertible if{}f $x_{\phi}$ belongs to $T(E)^*$.
\end{theorem}

Following \cite{Sil}, pag 98, we denote by $\mathcal{ELL}(R)$ the moduli space of elliptic curves with $\End(E)\simeq R$. Then (\cite{Sil}, II Proposition 1.2) $\mathcal{ELL}(R)$ is a $Cl(R)$-torsor.

\begin{definition}~
\begin{enumerate}\itemsep=2pt 
\item Let $\rho\in \A^*$ be an \textit{id{\`e}le}. Recall that the \textit{ideal of} $\rho$ (\cite{Sil} p.~119) is the fractional ideal
$$ (\rho):= \prod_{\mathfrak{p}} \mathfrak{p}^{ord_{\mathfrak{p}}\rho_{\mathfrak{p}}}\in \jc$$
(where the product is over all the primes of $R$) and the \textit{multiplication by}
$\rho$ map is defined as
$$ \rho: \jc \longrightarrow \jc, \hskip3mm \Lambda \longrightarrow (\rho)\Lambda.$$
By an abuse of notations, we use the same symbol for the corresponding map on $Cl(R)$
$$ \rho: Cl(R)  \longrightarrow Cl(R), \hskip3mm [\Lambda] \longrightarrow \rho[\Lambda]:=[(\rho)\Lambda],$$
and on $\mathcal{ELL}(R)$
$$ \rho: \mathcal{ELL}(R)  \longrightarrow \mathcal{ELL}(R), \hskip3mm E_{[\Lambda]} \longrightarrow E_{\rho [\Lambda]}.$$
\item We denote by $\ac$ ($\ac^*$) the set of (invertible) arithmetic $\K$-lattices and by $\lc $ 
($\lc^*$) the set of (invertible) $\K$-lattices modulo dilations. Fix an id{\`e}le
$\rho \in \As$ and a fractional ideal $\Lambda\in \jc$.  The \textit{multiplication by
$\rho$ map on $\frac{\K}{R}$} is defined by the commutativity of the following diagram (compare with \cite{Sil} p.~159)
 
 \begin{center}
 \begin{tikzpicture}[description/.style={fill=white,inner sep=4pt},normal line/.style={->,font=\footnotesize,shorten >=4pt,shorten <=4pt}]
 \matrix (m) [matrix of math nodes, row sep=2.5em, column sep=2.5em, text height=1.5ex, text depth=0.25ex]
{  \frac{\K}{R} & \frac{\K}{\rho R} \\
    \bigoplus_{\mathfrak{p}}\frac{\K_{\mathfrak{p}}}{R_{\mathfrak{p}}}  & \bigoplus_{\mathfrak{p}}\frac{\K_{\mathfrak{p}}}{\rho_{\mathfrak{p}} R_{\mathfrak{p}}} \\  };
 \path[normal line] (m-1-1) edge node[above] {$\rho$} (m-1-2);
 \path[normal line] (m-2-1) edge node[above] {$\oplus_{\mathfrak{p}} \rho_{\mathfrak{p}}\cdot$} (m-2-2);
 \path[normal line] (m-1-1) edge node[left] {} (m-2-1);
 \path[normal line] (m-1-2) edge (m-2-2);
 \end{tikzpicture}
\end{center}
and it can be viewed as an element of $ \Hom_{R}\big(\frac{\K}{R},\frac{\K}{\rho R}\big)\simeq \hat{R}$
(compare with \eqref{homisomorphism}), hence as a lattice $\Phi_{\rho}\in \ac$ (we also define $\Psi_{\rho}:=[\phi_{\rho}]\in \lc$).
Obviously such a lattice is invertible as the inverse of $\rho: \frac{\K}{R}\to \frac{\K}{\rho R}$ is provided by $\rho^{-1}\in \A^*$, 
thus $\Phi_{\rho}$ belongs to $ \ac ^*$ ($\Psi_{\rho}\in \lc^*$).
\end{enumerate}
\end{definition}

\begin{theorem}\label{arithKlattices}~
\begin{enumerate}\itemsep=3pt
\item The map $\Phi$ just defined 
$$
\Phi: \As \longrightarrow \ac^*, \hskip4mm \rho \longrightarrow \Phi_{\rho}
$$
is bijective.
\item   $\As \simeq \hat{R}\times_{\hat{R}^*}\ac^*\simeq \hat{R}\times_{\hat{R}^*}\As$. 
\end{enumerate}
\end{theorem}
\begin{proof}
(1) Surjectivity: the multiplication  map
$$ \rho: \As   \longrightarrow \jc, \hskip3mm \rho  \longrightarrow (\rho)R$$
is obviously surjective. Furthermore, any two invertible lattices in $\Hom_{R}\big(\frac{\K}{R},\frac{\K}{\rho R}\big)$ differ (by Theorem \ref{sesm}) by an element of $T(E)^*$, which is an $\hat{R}^*$-torsor
(compare with Section \ref{AECCM}).

Injectivity: if $\Phi_{\rho}=\Phi_{\rho'}$ then $\rho R=\rho'R:= \Lambda $ and $\frac{\rho}{\rho'}=id \in \Hom_{R}\big(\frac{\K}{\Lambda},\frac{\K}{\Lambda}\big)$, hence
$\frac{\rho}{\rho'}=1\in \As$ because of Theorem \ref{sesm}.
\par\noindent
(2) Just combine (1) with Proposition 4.6 of \cite{KMS}.
\end{proof}

\bigskip
Taking quotients by $\K^*$ in the theorem above we also have

\begin{corollary}\label{Klattices}~
\begin{enumerate}\itemsep=3pt
\item The map 
$$
\psi: \As/\K^* \longrightarrow \lc^*, \hskip4mm [\rho] \longrightarrow \psi_{\rho},
$$
defined on the \textit{id{\`e}le class group} $\As/\K^*$ (\cite{ANT} p.~142), is
bijective and the projection $\lc^* \rightarrow Cl(R)$ coincides with the usual Class Field map (\cite{ANT}, p.~224).
\item   $\lc \simeq \hat{R}\times_{\hat{R}^*}\lc^*\simeq \hat{R}\times_{\hat{R}^*}(\As/\K^*)$. 
\end{enumerate}
\end{corollary}
\bigskip

\begin{lemma}
\label{commensurability}
Consider an arithmetic  1-dimensional $\K$ lattice $(\Lambda,\phi)$. Any 1-dimensional $\K$ lattice which is commensurable to $(\Lambda,\phi)$ is arithmetic.
\end{lemma}
\begin{proof}
Assume that $(\Lambda', \psi)$ is commensurable to $(\Lambda,\phi)$. Then both the natural projections $\Lambda \rightarrow \Lambda \cap \Lambda'$, $\Lambda' \rightarrow \Lambda \cap \Lambda'$ have finite index hence the same happens for
$\Lambda + \Lambda'\rightarrow \Lambda \cap \Lambda'$. So there exists an integer $n$ s.t.~$n(\Lambda + \Lambda')\subset \Lambda \cap \Lambda'$. The thesis follows since 
$n\Lambda' \subset n(\Lambda + \Lambda')\subset \Lambda \cap \Lambda' \subset \K$.  
\end{proof}

\begin{notations} We denote by $S$ the set of \textit{non-archimedean places of $\K$} (\cite{LocF}, p.~189). For any $v\in S$ we denote by $\mathfrak{p}_v\subset R$ the prime ideal corresponding to $v$ and
we choose a \textit{prime element}  $\pi_v\in \mathfrak{p}_v $ (\cite{LocF}, p.~42).
\end{notations}

\begin{lemma}
\label{multiplication}
Consider two fractional ideals $\Lambda $, $\Lambda' $,  assume $\Lambda \subset \Lambda'$ and set 
$\Lambda=  \prod_{v\in S} \mathfrak{p}_v^{e(v)} \cdot \Lambda'$.
For any $\gamma \in T(\frac{\C}{\Lambda'})^*$ there exists $\delta \in T(\frac{\C}{\Lambda})^*$ s.t. the natural map  
$\Hom_R\big(\frac{\K}{R},\frac{\K}{\Lambda}\big)\to\Hom_R\big(\frac{\K}{R},\frac{\K}{\Lambda'}\big)$ induced by  projection is represented, via isomorphisms stated in \eqref{homisomorphism}, by the multiplication  by $\prod_{v\in S} \pi_v^{e(v)}$.
\end{lemma}
\begin{proof}
Fix $\delta' \in T(\frac{\C}{\Lambda})^*$ and assume that the map in $\Hom_R\big(\frac{\K}{\Lambda},\frac{\K}{\Lambda'}\big)$ induced by projection is represented by
$\alpha \in \hat{R}$ via \eqref{homisomorphism}.
Since $Ker(\alpha) \simeq \frac{\Lambda'}{\Lambda}\simeq \prod_{v\in S} \mathfrak{p}_v^{e(v)}$, we may assume
$\alpha = \prod_{v\in S} \pi_v^{e(v)}\alpha'$, with $\alpha' \in \hat{R}^*$. In order to conclude it suffices to choose $\delta =\delta' \cdot \alpha'$.
\end{proof}

We recall the following fact observed in the proof of Proposition 4.5 of \cite{KMS}:

\begin{lemma}\label{fact}
Two  invertible 1-dimensional  $\K$  lattices are commensurable if{}f they coincide.
\end{lemma}

Furthermore, we get

\begin{lemma}\label{fact'}
A non-invertible arithmetic 1-dimensional $\K$ lattice is commensurable to an invertible one.
\end{lemma}
\begin{proof}
Any arithmetic lattice is represented by means of \eqref{homisomorphism} by an element $\rho \in \hat{R}$ which can be written as $\rho= \alpha \cdot \rho'$, where
$\alpha$ is a suitable product of prime elements and $\rho'\in  \hat{R}^*$. Lemma \ref{multiplication} implies that $\rho'\in  \hat{R}^*$ can be interpreted as an invertible lattice for some $\Lambda'\subset \Lambda $.
\end{proof}

\begin{corollary}\label{invertiblelattices}~
 The set $\fc$ of arithmetic  1-dimensional $\K$  lattices modulo commensurability can be identified with $\bigcup_{\Lambda}T(E_{\Lambda})^*$.
\end{corollary}
\begin{proof} 
Just combine Lemmas \ref{fact} and \ref{fact'}.
\end{proof}

\begin{definition}
\label{groupoid}
Consider two fractional ideals $\Lambda$ and $\Lambda'$ and set $\Gamma =\Lambda + \Lambda'$. Denote by 
$p:\Hom_R\big(\frac{\K}{R},\frac{\K}{\Lambda}\big)\to \Hom_R\big(\frac{\K}{R},\frac{\K}{\Gamma}\big)$ and by 
$q:\Hom_R\big(\frac{\K}{R},\frac{\K}{\Lambda'}\big)\to\Hom_R\big(\frac{\K}{R},\frac{\K}{\Gamma}\big)$ the natural projection and consider the corresponding $\rho$ and $\rho'$ obtained by means of Lemma \ref{multiplication}. Then the lattices $(\Lambda, \phi)$ and $(\Lambda', \phi')$ are commensurable if{}f $p(\phi)=p'(\phi')$ so that the set of such commensurable lattices can be identified with the fiber product $\fc _{\Lambda, \Lambda'}\subset T(E_{\Lambda})\times T(E_{\Lambda'})$ arising from the Cartesian square:
 \begin{center}
 \begin{tikzpicture}[description/.style={fill=white,inner sep=4pt},normal line/.style={->,font=\footnotesize,shorten >=4pt,shorten <=4pt}]
 \matrix (m) [matrix of math nodes, row sep=3em, column sep=3.5em, text height=1.5ex, text depth=0.25ex]
 { \fc _{\Lambda , \Lambda '} & T(E_{\Lambda}) \\
   T(E_{\Lambda'})  & T(E_{\Gamma}) \\ };
 \path[normal line] (m-1-1) edge node[left] {$q$} (m-2-1);
 \path[normal line] (m-1-2) edge node[right] {$p$} (m-2-2);
 \path[normal line] (m-1-1) edge (m-1-2);
 \path[normal line] (m-2-1) edge (m-2-2);
 \end{tikzpicture}
 \end{center}
\end{definition}

In the notations of Definition \ref{groupoid},  Corollary \ref{invertiblelattices} implies the following:

\begin{theorem} 
\label{groupoidthm}
Fix a set of representatives $\Lambda _i$ of $Cl(R)$, 
 $1\leq i \leq \sharp Cl(R)$. Set 
$\ec= \{1,2,\dots ,\sharp Cl(R)\}\times \mathcal{J}$. Then
the groupoid of commensurability modulo dilations can be identified with:
$$\bigcup_{(\Lambda_i,\Lambda)\in\ec  }\fc _{\Lambda_i, \Lambda}\subset \bigcup_{(\Lambda_i,\Lambda)\in\ec }
T(E_{\Lambda_i})\times T(E_{\Lambda}).$$
\end{theorem}
\begin{proof}
If $\Lambda $ and $\Lambda'$ are commensurable,  Lemma \ref{commensurability} implies that there exists $k\in \mathbb{C}$ s.t.~$k\Lambda = \Lambda_i$ and $k\Lambda'\in Cl(R)$. Then we may conclude by means of Corollary \ref{invertiblelattices} and Definition \ref{groupoid}.
\end{proof}
We recall the following  (compare with \cite{Theta3}, Definition 4.21 and Lemma 4.22 and remind that we are working with fields of characteristic $0$):      

\begin{definition}\label{Qisogeny} Consider two elliptic curves $E$ and $E'$.
A $\Q$-\textit{isogeny} from  $E$ to $E'$ is a triple $(Z, f_1, f_2)$, where $Z$ is an elliptic curve and 
$f_1: Z\to E$, $f_2:Z \to E'$ are isogenies. Two $\Q$-isogenies $(Z, f_1, f_2)$, $(W, g_1, g_2)$ are \textit{equivalent} if there is an elliptic curve $C$ and isogenies
$a: C\to Z$, $b:C\to W$ so that $f_i\circ a=g_i\circ b$, $i=1$, $2$.  Any $\Q$-isogeny $\alpha =(Z, f_1, f_2)$ induces an isomorphism $V(\alpha):V(E) \to V(E')$ and equivalent $\Q$-isogenies induce the same map.

Equivalent classes of $\Q$-isogenies from $E$ to $E'$  for a fixed $E$ are in $1$-$1$ correspondence with open subgroups of $V(E)$ (the open subgroup 
$V(\alpha)^{-1}T(E')\subset V(E)$ corresponding to the $\Q$-isogeny $\alpha =(Z, f_1, f_2)$).  
\end{definition}

\begin{notations}\label{mapslattices}~
\begin{enumerate}\itemsep=3pt
\item Consider two fractional ideals $\Lambda$,  $\Lambda'$ of $R$, then there is an obvious
$\Q$-\textit{isogeny} $\alpha_{\Lambda, \Lambda'}:=(E_{\Lambda \cap \Lambda'}, p_1, p_2)$ from $E_{\Lambda}$ to $E_{\Lambda'}$, where $p_1$ and $p_2$ are the  natural     projections. Recalling Definition \ref{Qisogeny},   we have a natural
morphism 
$$V(\alpha_{\Lambda, \Lambda'})^{-1}: T(E_{\Lambda'})^*\rightarrow  V(E_{\Lambda}).$$
\item Gluing the morphisms $V(\alpha_{\Lambda, \Lambda'})^{-1}$ just defined  
we get a    natural map from the  set of arithmetic  1-dimensional $\K$  lattices modulo commensurability,  $\fc$
(compare with Corollary \ref{invertiblelattices}),   to $V(E_{\Lambda})$:
$$\rho_{\Lambda}:\fc = \bigcup_{\Lambda'}T(E_{\Lambda'})^*\rightarrow V(E_{\Lambda}), \hskip4mm \rho_{\Lambda}\mid_{T(E_{\Lambda'})^*}:=V(\alpha_{\Lambda, \Lambda'})^{-1}.$$
\item With notations as in Definition \ref{groupoid} and in Theorem \ref{groupoidthm} we can define a natural map
$$\xi:  \bigcup_{(\Lambda_i,\Lambda)\in\ec }T(E_{\Lambda_i})\times T(E_{\Lambda}) \longrightarrow  \bigcup_{\Lambda_i }  V(E_{\Lambda_i}),  
$$
acting on each $T(E_{\Lambda_i})\times T(E_{\Lambda}) $ as
$$\xi: T(E_{\Lambda_i})\times T(E_{\Lambda}) \rightarrow V(E_{\Lambda_i}), \hskip3mm (x,y)\rightarrow   V(\alpha_{\Lambda_i, \Lambda})^{-1}(y)\cdot x.
$$
By Theorem \ref{groupoidthm}, such a map restricts to a map from  the groupoid of commensurability modulo dilations 
$$\xi: \bigcup_{(\Lambda_i,\Lambda)\in\ec  }\fc _{\Lambda_i, \Lambda}\longrightarrow  \bigcup_{\Lambda_i }  V(E_{\Lambda_i}), \hskip3mm
\xi\mid_{\fc _{\Lambda_i, \Lambda}}: (x,y)\rightarrow   V(\alpha_{\Lambda_i, \Lambda})^{-1}(y)\cdot x. $$
\end{enumerate}
\end{notations}

\section{The Adelic Heisenberg group.}\label{AHGsec}

\begin{notations}
\label{notationsAHG}~
\begin{enumerate}\itemsep=3pt
\item
Let $L\in Pic(E)$  be  a line bundle on an elliptic curve $E$ with projection map
 $p_L:  L\longrightarrow E$. We  assume the curve $E$   to be  uniformized as
$$\pi_{\Lambda}: \C  \longrightarrow  \frac{\mathbb C }{\Lambda}\simeq E.$$ 
By \cite{CAV}, Proposition 2.1.6 and Lemma 2.1.7,  the first Chern class  of $L$, $c_1(L)$,   can be identified with an alternating form 
$$\C\times \C \longrightarrow \mathbb{R}, \hskip2mm (z,w)\longrightarrow r\Im (z\overline{w}), $$
s.t. $r\Im (\Lambda\overline{\Lambda}) \subset \mathbb{Z}$.
\par\noindent
If we fix  an integral basis of the lattice $\Lambda $: $\lambda := l_x+il_y$, $\mu :=m_x+im_y$ s.t. $\Im (\frac{\lambda}{\mu})>0$
 then, in order to have 
$r\Im (\Lambda\overline{\Lambda}) \subset \mathbb{Z}$,  we must have   
\begin{equation}
\label{r}
r=\frac{d}{l_x m_y- l_y m_x}
\end{equation}
where $d$ is the \textit{degree}  of    $L$ (compare with \cite{Hus}, \S 12.2).
\item 
As in \cite{CAV}, \S 2.4,   we denote by $t_x$ the \textit{translation by}    $x\in E$ and     define:
$$
K(L):=\{ x\in E \hskip1mm \mid \hskip1mm t_x^*L \simeq L \}, \hskip2mm \Lambda(L):=\pi_{\Lambda}^{-1}(K(L)).
$$
By \cite{CAV}, \S 2.4,
\begin{equation}
\label{Lambda}
\Lambda(L)=\{ z\in \mathbb{C} \hskip1mm \mid \hskip1mm r\Im(z\overline{\Lambda}) \subset \mathbb{Z} \}\simeq \frac{1}{d}\Lambda, 
\end{equation} 
hence we have:   
\begin{equation}
\label{K}
 K(L)\simeq\frac{1/d\Lambda}{\Lambda}\simeq  E[d].  
\end{equation}
\end{enumerate}  
\end{notations}

\begin{remark}
\label{a2}
Assume that  $End(E)\simeq R$ and fix $a\in R$. Then 
$E\stackrel{a}{\longrightarrow} E$ has degree $\vert a \vert ^2$ (\cite{Sil0} \S 5, Proposition 2.3, \cite{Hus} \S 12, Proposition 1.3) 
hence \eqref{K} implies: 
\begin{equation}
\label{KaL}
K(a^*L)=E[\vert a \vert ^2 d]
\end{equation}
\end{remark}

\begin{definition}
\label{Heisenberg}
Let $x\in E$, then a biholomorphic map $\phi: L \longrightarrow L$ is called an \textit{automorphism of} $L$ \textit{over} $x$, if the diagram
\begin{center}
 \begin{tikzpicture}[description/.style={fill=white,inner sep=4pt},normal line/.style={->,font=\footnotesize,shorten >=4pt,shorten <=4pt}]
 \matrix (m) [matrix of math nodes, row sep=3em, column sep=3.5em, text height=1.5ex, text depth=0.25ex]
 { L & L \\ Y & X \\ };
 \path[normal line] (m-1-1) edge node[left] {$p_L$} (m-2-1);
 \path[normal line] (m-1-2) edge node[right] {$p_L$} (m-2-2);
 \path[normal line] (m-1-1) edge node[above] {$\phi$} (m-1-2);
 \path[normal line] (m-2-1) edge node[above] {$t_x$} (m-2-2);
 \end{tikzpicture}
 \end{center}
\noindent
commutes (\cite{CAV}, 6.1). This of course forces $x$ to belong to $K(L)$. We denote by  $\mathcal{G}(L)_x$ the set of automorphisms of $L$ over $x$ and set 
$$\mathcal{G}(L):= \bigcup_{x\in K(L)}\mathcal{G}(L)_x.$$
Then $\mathcal{G}(L)$ is a group via composition of automorphisms, the \textit{Heisenberg group} of $L$, which is known to be a central extension of $K(L)$:
$$
0\longrightarrow \mathbb{C}^* \longrightarrow \mathcal{G}(L) \stackrel{g_L}{\longrightarrow} K(L)\longrightarrow 0  
$$
(compare with \cite{Theta3}, \S 1 and with  \cite{CAV}, \S 6). As for  any central extension of an abelian group we can define   
 a \textit{commutator map}:
$$
e^L: K(L)\times K(L)\longrightarrow \mathbb{C}^*
$$
$$
e^L(x,y):=\tilde{x}\tilde{y}\tilde{x}^{-1}\tilde{y}^{-1},   
$$
where $g_L(\tilde{x})=x$ and $g_L(\tilde{y})=y$.
If $z$, $w\in \mathbb{C}$ are chosen in such a way that $\pi_{\Lambda}(z)=x$ and $\pi_{\Lambda}(w)=y$, then (\cite{CAV}, Proposition 6.3.1)   
we have
\begin{equation}\label{e}
e^L(x,y)=\exp\{-2\pi ir\Im(z\overline{w}) \}\in \boldsymbol{\mu},
\end{equation}
which obviously takes values in the \textit{group of  roots of unity} $\boldsymbol{\mu}$.
\end{definition} 

\begin{remark}\label{Weilp}
Combining \eqref{r}, \eqref{K} and  \eqref{e} one can recognize in $e^L$ the inverse of the \textit{Weil Pairing} on $E[d]$ (\cite{Hus}, \S 12):
$$e^L(x,y)=W(x,y)^{-1}$$  
\end{remark}

From now on we assume that $\Lambda$ is a fractional ideal of $\K$ (so $End(E)\simeq R$).

\begin{lemma}
\label{lemmaequivalence}
Assume that $ay=x \in E$, then we have:
\begin{enumerate}\itemsep=3pt
\item if $x\in K(L)$ then $y\in K(a^*L)$;
\item if $y\in K(a^*L)$ then
$$x\in K(L) \Leftrightarrow e^{a^*L}\mid_{A_x}\equiv 1, $$
where $A_x:=a^{-1}(x)\times  a^{-1}(x)$.
\end{enumerate}
\end{lemma}
\begin{proof}
(1) Combining \eqref{K} with Remark \ref{a2} we find
$$
K(L)=E[d]\subset E[\vert a\vert^2d]=K(a^*L).$$
\par\noindent  
(2) We have a commutative diagram:
 \begin{center}
 \begin{tikzpicture}[description/.style={fill=white,inner sep=4pt},normal line/.style={->,font=\footnotesize,shorten >=4pt,shorten <=4pt}]
 \matrix (m) [matrix of math nodes, row sep=3em, column sep=3.5em, text height=1.5ex, text depth=0.25ex]
 { \mathbb{C} & \mathbb{C} \\ E & E \\ };
 \path[normal line] (m-1-1) edge node[left] {$\pi_{\Lambda}$} (m-2-1);
 \path[normal line] (m-1-2) edge node[right] {$\pi_{a\Lambda}$} (m-2-2);
 \path[normal line] (m-1-1) edge node[above] {id} (m-1-2);
 \path[normal line] (m-2-1) edge node[above] {a} (m-2-2);
 \end{tikzpicture}
 \end{center}
showing that (recall \eqref{Lambda})  
\begin{equation}\label{aLambda}
\Lambda(L)=\{ z\in \mathbb{C} \hskip1mm \mid \hskip1mm r\Im(z\overline{\Lambda})\subset \mathbb{Z} \}\subset
\{ z\in \mathbb{C} \hskip1mm \mid \hskip1mm r\Im(z\overline{a\Lambda}) \subset \mathbb{Z}\} = \Lambda (a^*L).
\end{equation} 
Choose a set of representatives of $\frac{\Lambda}{a\Lambda}$: $v_i\in \Lambda $, $1\leq i \leq \vert a\vert^2$.
\par\noindent
If $z\in \Lambda (a^*L)$ then \eqref{aLambda} shows that
\begin{equation}\label{v_i}
z\in \Lambda(L) \hskip1mm \Leftrightarrow \hskip1mm r\Im(z\overline{v_i}) \in  \mathbb{Z}, \hskip1mm  
1\leq i \leq \vert a\vert^2.
\end{equation}
If $z\in \C$ projects on  both $x$ and $y$, and if  $e^{a^*L}\mid_{A_x}\equiv 1 $ then \eqref{e} implies 
$$r\Im(z\overline{v_i})  = r\Im(z(\overline{z+v_i}))\in  \mathbb{Z}, \hskip1mm 1\leq i \leq \vert a\vert^2.$$
Conversely, if  $r\Im(z\overline{v_i})\in  \mathbb{Z}$, $1\leq i \leq \vert a\vert^2$, then we have
$$r\Im((z+v_i)(\overline{z+v_j})) \in  \mathbb{Z}, \hskip1mm 1\leq i,  j \leq \vert a\vert^2$$
hence $e^{a^*L}\mid_{A_x}\equiv 1 $.  
\end{proof}

\par\noindent
Consider now the pull-back commutative diagram
\begin{center}
 \begin{tikzpicture}[description/.style={fill=white,inner sep=4pt},normal line/.style={->,font=\footnotesize,shorten >=4pt,shorten <=4pt}]
 \matrix (m) [matrix of math nodes, row sep=3em, column sep=3.5em, text height=1.5ex, text depth=0.25ex]
 { a^*L & L \\ E & E \\ };
 \path[normal line] (m-1-1) edge node[left] {$q$} (m-2-1);
 \path[normal line] (m-1-2) edge node[right] {$p$} (m-2-2);
 \path[normal line] (m-1-1) edge node[above] {$\tilde{a}$} (m-1-2);
 \path[normal line] (m-2-1) edge node[above] {$a$} (m-2-2);
 \end{tikzpicture}
 \end{center}
 where we put $q:=p_{a^*L}$ and $p:=p_L$ in order to ease the notations.  
The following Lemma relates the automorphisms of $L$ over a point $x$ to the automorphisms of $a^*L$ over any point $y$, s.t. $ay=x$.

\begin{lemma}
\label{covering}
Assume that $ay=x \in E$ and fix $\phi \in \gc(L)_x$. Then there exists a unique $\psi \in \gc(a^*L)_y$ s.t. $\tilde{a}\circ \psi=\phi \circ \tilde{a}$.
\end{lemma}
\begin{proof}
It is an easy application of the universal property of pull-back diagrams.\par\noindent  
Since $ay=x$, we also have $t_x\circ a=a\circ t_y$, and
$$a\circ t_y \circ q=t_x \circ a \circ q=t_x \circ p\circ \tilde{a}=p\circ \phi \circ \tilde{a}
$$
hence
$f:=t_y \circ q: a^*L\longrightarrow E$ and $g:=\phi \circ \tilde{a}: a^*L\longrightarrow L$ satisfy
$a\circ f=p \circ g$.
By the universal property of pull-back diagrams there exists a unique $\psi: a^*L \longrightarrow a^*L$ such that
$f=q\circ \psi$ and $g=\tilde{a} \circ \psi$, which amounts to say   $\psi \in \gc(a^*L)_y$ and 
 $\tilde{a}\circ \psi =\phi \circ \tilde{a}$. 
\end{proof}

\begin{definition}
\label{lifting}
Following \cite{Theta3}, Definition 4.5 we will say that the  automorphism 
$\psi\in  \gc(a^*L)_y $ defined in Lemma \ref{covering}
\textit{covers} $\phi $  \textit{over}  $t_y $. We will also say that
$\phi $ 
\textit{is covered by} $\psi $  \textit{over}  $t_y $.
\end{definition}

\begin{theorem}
\label{isomorphismThm} Assume that $ay=x \in E$, and that $y\in K(a^*L)$, TFAE:
\begin{enumerate}\itemsep=3pt
\item  $e^{a^*L}\mid_{A_x}\equiv 1$;
\item $x\in K(L)$;
\item for any $\psi\in  \gc(a^*L)_y $ there exists $\phi \in \gc(L)_x  $ s.t. $\tilde{a}\circ \psi =\phi \circ \tilde{a}$;
\item there are $\psi\in  \gc(a^*L)_y $ and  $\phi \in \gc(L)_x$ s.t. $\tilde{a}\circ \psi =\phi \circ \tilde{a}$.
\end{enumerate}
\end{theorem}
\begin{proof}
(1) $\Leftrightarrow$ (2) follows by Lemma \ref{lemmaequivalence}.
\par\noindent
(2) $\Rightarrow$ (3) follows by Lemma \ref{covering}.
\par\noindent
(3) $\Rightarrow$ (4) is obvious.
\par\noindent
(4) $\Rightarrow$ (2) follows because if $\phi \in \gc(L)_x$ then $x\in K(L)$.
\end{proof}

\begin{remark}
\label{isomorphism}~
\begin{enumerate}\itemsep=3pt
\item
Under the hypotheses of either
Lemma \ref{covering} or Theorem \ref{isomorphismThm} we have   a bijective map 
$$  \gc(L)_x \leftrightarrow \gc(a^*L)_y
$$ 
\par\noindent 
defined by the relation
$$ 
\phi \leftrightarrow \psi \hskip2mm \Leftrightarrow \hskip2mm \tilde{a}\circ \psi =\phi \circ \tilde{a}.
$$
\item It is immediate to check that the bijective maps just defined preserve the composition of automorphisms,  so they  glue to give  a group morphism
$$\delta_{a,L}:\gc(a^*L)\mid_{a^{-1}(K(L))}\longrightarrow \gc(L).$$
\end{enumerate}
\end{remark}

\begin{definition} Let $x=(x_a)_{a\in R}\in V(E)$ so that $x_1\in E_{tor}$.
Set 
$$I_x:=\{a\in R \mid \hskip1mm a x_1\in K(L) \}$$ 
and 
$$J_x:=\{a\in R \mid \hskip1mm  x_a\in K(a^*L) \}.$$  
\end{definition}

\begin{lemma}\label{IeJ} Let $x=(x_a)_{a\in R}\in V(E)$ so that $x_1\in E_{tor}$ and choose $k\in \K$ s.t. $\pi_{\Lambda}(k)=x_1$.
Then we have $$I_x=\overline{J_x}\simeq \frac{1}{kd}\Lambda \cap R.$$
\end{lemma}
\begin{proof}
$I_x\simeq \frac{1}{kd}\Lambda \cap R:$
by \eqref{Lambda}, for any $a\in R$ we have:
$$ax_1\in K(L)\hskip1mm \Leftrightarrow \hskip1mm ak\in \frac{1}{d}\Lambda, \hskip1mm \Leftrightarrow \hskip1mm a\in \frac{1}{kd}\Lambda \cap R.$$
$\overline{J_x} \simeq \frac{1}{kd}\Lambda \cap R:$
Choose $k_a\in \K$ s.t. $ak_a=k$ and $\pi_{\Lambda}(k_a)=x_a$. Combining \eqref{Lambda} and \eqref{K} with Remark \ref{a2}, for any $a\in R$ we have:
$$x_a\in K(a^*L) \hskip1mm \Leftrightarrow \hskip1mm k_a\in \frac{1}{a\overline{a}d}\Lambda, \hskip1mm \Leftrightarrow \hskip1mm k=ak_a\in \frac{1}{\overline{a}d}\Lambda, \hskip1mm \Leftrightarrow \hskip1mm \overline{a}\in \frac{1}{kd}\Lambda \cap R. $$
\end{proof}

\begin{remark}
If $x\in V(E)$ and  $a, b\in J_x$ then by Remark \ref{isomorphism} we have  canonical isomorphisms:
$$
\mathcal{G}(a^*L)_{x_a}\simeq \mathcal{G}((ab)^*L)_{x_{ab}} \simeq \mathcal{G}(b^*L)_{x_b}
$$
in the sense that any $\phi_a\in \mathcal{G}(a^*L)_{x_a}$ ($\phi_b\in \mathcal{G}(b^*L)_{x_b}$)  is covered by a unique element  $\phi_{ab}\in\mathcal{G}((ab)^*L)_{x_{ab}}$ over $t_{x_{ab}}$:
$$\tilde{b}\circ \phi_{ab} =\phi_{a} \circ \tilde{b}, \hskip2mm \tilde{a}\circ \phi_{ab} =\phi_{b} \circ \tilde{a}.
$$
\end{remark}

By the previous remark we can give the following

\begin{definition}\label{AHG}~
\begin{enumerate}\itemsep=3pt
\item We denote by $\hat{\mathcal{G}}(L)$ the \textit{adelic Heisenberg group} i.e.~the group   of
$\alpha =(x,  (\phi_a)_{a\in J_{x}})$ s.t.~$x\in V(E)$, $\phi_a \in \mathcal{G}(a^*L)_{x_a}$ with $\phi_{ab}$ covering $\phi_a$ over $t_{x_{ab}}$,  $\forall a\in J_{x}$.
Any  collection of elements $(\phi_a)_{a\in J_{x}} $, $\phi_a \in \mathcal{G}(a^*L)_{x_a}$ with $\phi_{ab}$ covering $\phi_a$ over $t_{x_{ab}}$  will be called a \textit{coherent system of automorphisms} defined over $x$.
\item Using the (rather cumbersome) notations of Remark \ref{isomorphism} (2), if $(\phi_a)_{a\in J_{x}} $ is a coherent system of automorphisms then
$$\delta_{b,a^*L}(\phi_{ab})=\phi_a, \hskip2mm \forall a\in J_x, \hskip1mm \forall b\in R. $$
\item Consider the closed subgroup $\vc_a:= \nu_a^{-1}(K(a^*L))\subset V(E)$. Obviously $a\in J_x$ $\forall x\in \vc_a$ so it is well defined a group morphism  
$$\gamma_a: \hat{\mathcal{G}}(L)\mid_{\vc_a}\longrightarrow \gc(a^*L), \hskip4mm (x,  (\phi_b)_{b\in J_{x}})\longrightarrow\phi_a. $$
\end{enumerate}
\end{definition}

\begin{theorem}\label{exactseq}~
\begin{enumerate}\itemsep=3pt
\item We have an exact sequence:
$$1\rightarrow \mathbb{C}^* \rightarrow \hat{\mathcal{G}}(L) \stackrel{g}{\rightarrow} V(E) \rightarrow 0.$$
\item $$\hat{\mathcal{G}}(L)\mid_{\vc_a}\simeq \nu_a^*\gc(a^*L).$$
\item $$\hat{\mathcal{G}}(L)\mid _{T(E)}\simeq \nu_1^*\C,$$
so there exists a morphism $\sigma ^L:T(E) \rightarrow \hat{\mathcal{G}}(L)$ providing a section of the restriction to $T(E)$ of the sequence above ($\sigma(x)$ is defined as the unique element of $\hat{\mathcal{G}}(L)$ lifting the identity).
\end{enumerate}
\end{theorem}
\begin{proof}
(1) By Lemma \ref{IeJ} the map $\hat{\mathcal{G}}(L) \rightarrow V(E)$ is surjective because $\overline{\frac{1}{kd}\Lambda \cap R}= J_x\not=\emptyset$, $\forall x\in V(E)$, so   there are  coherent systems of automorphisms defined over any $x\in V(E)$. Furthermore, for any $x\in V(E)$,  Remark \ref{isomorphism} implies that any coherent system of automorphisms  $(\phi_a)_{a\in J_{x}} $, $\phi_a \in \mathcal{G}(a^*L)_{x_a}$, is uniquely determined by any of its  components, so we have:
$$ \hat{\mathcal{G}}(L) _{x}\simeq \mathcal{G}(a^*L)_{x_a}\simeq \mathbb{C}^*, \hskip2mm \forall a\in J_x.$$
(2) Consider    the pull-back  diagram
\begin{center}
 \begin{tikzpicture}[description/.style={fill=white,inner sep=4pt},normal line/.style={->,font=\footnotesize,shorten >=4pt,shorten <=4pt}]
 \matrix (m) [matrix of math nodes, row sep=3em, column sep=3.5em, text height=1.5ex, text depth=0.25ex]
 { \nu_a^*\gc(a^*L) & \gc(a^*L) \\ \vc_a & K(a^*L) \\ };
 \path[normal line] (m-1-1) edge node[left] {} (m-2-1);
 \path[normal line] (m-1-2) edge node[right] {$g_{a^*L}$} (m-2-2);
 \path[normal line] (m-1-1) edge node[above] {} (m-1-2);
 \path[normal line] (m-2-1) edge node[above] {$\nu_a$} (m-2-2);
 \end{tikzpicture}
 \end{center}
Obviously the morphisms  $g_{a^*L}\circ\gamma_a$ and $\nu_a\circ g$ coincide on  $\hat{\mathcal{G}}(L)\mid_{\vc_a}$ so,   by the universal property of pull-back,  there is a unique morphism 
$\iota:\hat{\mathcal{G}}(L)\mid_{\vc_a} \longrightarrow \nu_a^*\gc(a^*L)$ s.t.
\begin{center}
 \begin{tikzpicture}[description/.style={fill=white,inner sep=4pt},normal line/.style={->,font=\footnotesize,shorten >=4pt,shorten <=4pt}]
 \matrix (m) [matrix of math nodes, row sep=3em, column sep=3.5em, text height=1.5ex, text depth=0.25ex]
 { \hat{\mathcal{G}}(L)\mid_{\vc_a} & \nu_a^*\gc(a^*L) \\ \vc_a & \vc_a  \\ };
 \path[normal line] (m-1-1) edge node[left] { } (m-2-1);
 \path[normal line] (m-1-2) edge node[right] {} (m-2-2);
 \path[normal line] (m-1-1) edge node[above] {$\iota$} (m-1-2);
 \path[normal line] (m-2-1) edge node[above] {id} (m-2-2);
 \end{tikzpicture}
 \end{center}
commutes.
Finally, $\iota$ is an isomorphism because it is bijective on the fibers (both isomorphic to $\C^*$) of the vertical maps  of the last diagram.  
\par\noindent
(3) This  is a particular case of (2).   
\end{proof}

\begin{lemma}
\label{ewelldefined}
Let $(x,  (\phi_a)_{a\in J_{x}}), (y,  (\phi_a)_{a\in J_{y}})\in \hat{\mathcal{G}}(L)$ and $a, b\in J_x\cap J_y$, then 
$$ e^{a^*L}(x_a, y_a)=e^{b^*L}(x_b, y_b)$$
\end{lemma}
\begin{proof}
We argue as in the proof of Lemma \ref{lemmaequivalence}.
We have  commutative diagrams
 \begin{center}
 \begin{tikzpicture}[description/.style={fill=white,inner sep=4pt},normal line/.style={->,font=\footnotesize,shorten >=4pt,shorten <=4pt}]
 \matrix (m) [matrix of math nodes, row sep=3em, column sep=3.5em, text height=1.5ex, text depth=0.25ex]
 { \mathbb{C} & \mathbb{C} \\ E & E \\ };
 \path[normal line] (m-1-1) edge node[left] {$\pi_{b\Lambda}$} (m-2-1);
 \path[normal line] (m-1-2) edge node[right] {$\pi_{ab\Lambda}$} (m-2-2);
 \path[normal line] (m-1-1) edge node[above] {id} (m-1-2);
 \path[normal line] (m-2-1) edge node[above] {a} (m-2-2);
 \end{tikzpicture}
 \end{center}
 and
 \begin{center}
 \begin{tikzpicture}[description/.style={fill=white,inner sep=4pt},normal line/.style={->,font=\footnotesize,shorten >=4pt,shorten <=4pt}]
 \matrix (m) [matrix of math nodes, row sep=3em, column sep=3.5em, text height=1.5ex, text depth=0.25ex]
 { \mathbb{C} & \mathbb{C} \\ E & E \\ };
 \path[normal line] (m-1-1) edge node[left] {$\pi_{a\Lambda}$} (m-2-1);
 \path[normal line] (m-1-2) edge node[right] {$\pi_{ab\Lambda}$} (m-2-2);
 \path[normal line] (m-1-1) edge node[above] {id} (m-1-2);
 \path[normal line] (m-2-1) edge node[above] {b} (m-2-2);
 \end{tikzpicture}
 \end{center}
 so if we choose $z, v\in \C$ s.t. 
 $$\pi_{ab\Lambda}(z)=x_{ab}, \hskip2mm \pi_{ab\Lambda}(v)=y_{ab}$$ 
 then 
 $$\pi_{a\Lambda}(z)=x_{a}, \hskip1mm \pi_{a\Lambda}(v)=y_{a}, \hskip1mm \pi_{b\Lambda}(z)=x_{b}, \hskip2mm \pi_{b\Lambda}(v)=x_{b}.$$
Finally, \eqref{e} implies
$$e^{a^*L}(x_a, y_a)=\exp\{-2\pi ir\Im(z\overline{v}) \}=e^{b^*L}(x_b, y_b).
$$
\end{proof}

\begin{remark} Similarly as above (compare with Definition \ref{Heisenberg}),  we can define   
 a \textit{commutator map}:
$$\tilde{e}^L: V(E)\times V(E) \rightarrow \mathbb{C}^1, \hskip5mm \tilde{e}^L(x, x')=yy'y^{-1}y'^{-1}$$ where
 $y, y'\in \hat{\mathcal{G}}(L)$ are chosen in such a way that $g(y)=x, g(y')=x'$. By Lemma \ref{ewelldefined} we have  
\begin{equation}\label{etilde} \tilde{e}^L(x, x')=e^{a^*L}(x_a, x'_a)\in \boldsymbol{\mu}, \hskip1mm \forall a\in J_x\cap J_{x'}. 
\end{equation}
with values in the group of the roots of unity.
\end{remark}

\begin{theorem}
Let $x_a\in E_{tor}$ and set $$\tilde{A}_{x_a}:=\nu_a^{-1}(x_a)\times \nu_a^{-1}(x_a)\subset V(E)\times V(E).$$    Then the restriction 
$\tilde{e}^L(\cdot, \cdot)\mid_{\tilde{A}_{x_a}}$ is constant iff $x_a \in K(a^*L)=E[\vert a \vert ^2 d]$.
\end{theorem}
\begin{proof}
If $x_a \in K(a^*L)=E[\vert a \vert ^2 d]$ then \ref{etilde} implies
$$\tilde{e}^L(x,x')=e^{a^*L}(x_a, x_a)=1, \hskip1mm \forall (x,x')\in \tilde{A}_{x_a}.  
$$
Conversely, assume that $\tilde{e}^L(\cdot, \cdot)\mid_{\tilde{A}_{x_a}}\equiv 1$ and fix $b\in R$ s.t. $\overline{b}x_a\in K(a^*L)$.  By Lemma \ref{IeJ}
$$x_{ab}\in K((ab)^*L),\hskip1mm \forall x\in  \nu_a^{-1}(x_a)$$ which means that  
\begin{equation}\label{Kab} ab\in J_x, \hskip2mm \forall x\in  \nu_a^{-1}(x_a).        
\end{equation}
Consider $r, s\in b^{-1}(x_a)$ and choose $x, x'\in \nu_a^{-1}(x_a)$ s.t. $x_{ab}=r$ and $x_{ab}'=s$.   
Then $(x,x')\in  \tilde{A}_{x_a}$ and combining  \ref{etilde} with the hypothesis $\tilde{e}^L(\cdot, \cdot)\mid_{\tilde{A}_{x_a}}\equiv 1$
we find
$$e^{(ab)^*L}(r,s)=\tilde{e}^L(x,x')=1, \hskip1mm \forall r, s\in b^{-1}(x_a).
$$
We are done by means of Theorem \ref{isomorphismThm}.    
\end{proof}

\begin{remark}
Despite    the strong link with the Weil Pairing noticed in Remark \ref{Weilp}, last Theorem shows that the commutator map $\tilde{e}^L(\cdot, \cdot)\mid_{\tilde{A}_{x_a}}$ cannot be constant if $x_a \not\in K(a^*L)$. This proves that the bilinear form $\tilde{e}^L(\cdot, \cdot)$ on $V(E)$ \textit{cannot be induced by any bilinear form on $E_{tor}$} (such as the Weil pairing) via projection
$$\nu_a\times \nu_a: V(E)\times V(E)\rightarrow E_{tor}\times E_{tor}.$$
\end{remark}
\bigskip 

Assume now that $L$ is \textit{symmetric}: $(-1)^* L\simeq L$. Correspondingly  we have an involution
$i: \mathcal{G}(L)\to \mathcal{G}(L)$.
By an abuse of notations, we use the same symbol for the corresponding morphism of $\hat{\mathcal{G}}(L)$:
$$ i: \hat{\mathcal{G}}(L) \rightarrow \hat{\mathcal{G}}(L).$$
We recall  the following (\cite{Theta3}, p.~58):

\begin{definition}\label{tau}
Fix $x\in V(E)$, $y\in \hat{\mathcal{G}}(L)$ s.t.~$2g(y)=x$. Then $\tau (x)\in \hat{\mathcal{G}}(L)$ defined as $\tau (x):=y i(y)^{-1}$ does not depend of the choice of $y$, so we have a map:
$$ \tau : V(E)\rightarrow \hat{\mathcal{G}}(L),$$
providing a section of the exact sequence of Theorem \ref{exactseq}.
\end{definition}

\begin{proposition}
\label{notmorphism}
$$\tau (x)\tau (y)=\tilde{e}^L(\tfrac{x}{2},y)\tau (x+y)$$
\end{proposition}
\begin{proof}
Observe that, for any pair $x'\in \hat{\mathcal{G}}(L)$, $y'\in \hat{\mathcal{G}}(L)$ s.t.~$2g(x')=x$ and $2g(y')=y$ we have 
 $x'y'=\tilde{e}^L(\tfrac{x}{2},\tfrac{y}{2})y'x'$, 
$2g (i(x')^{-1})=x$ and $2g(i(y')^{-1})=y$ (and so $i(x')i(y')=\tilde{e}^L(\tfrac{x}{2},\tfrac{y}{2})i(y')i(x')$). Then we get
\begin{multline*}
\tau(x)\tau(y)=x'i(x')^{-1}y'i(y')^{-1}=\tilde{e}^L(\tfrac{x}{2}, \tfrac{y}{2})x'y'i(x')^{-1}i(y')^{-1} \\
=\tilde{e}^L(\tfrac{x}{2},y)x'y'i(y')^{-1}i(x')^{-1}=\tilde{e}^L(\tfrac{x}{2},y)x'y'i(x'y')^{-1}=\tilde{e}^L(\tfrac{x}{2},y)\tau (x+y).
 \qedhere
\end{multline*}
\end{proof}

\section{The Main Theorem of Complex Multiplication for Adelic Curves}

In this section we are going to study the behavior of the Adelic Heisenberg Group previously defined    under the action of field automorphisms of $\C$. We begin with some technical Lemmas which will be needed in the following.  

\begin{lemma}
\label{extensionsigma}
Consider two normalized (\cite{Sil} \S II, Proposition 1.1)   elliptic curves $E$, $E'$ s.t. $End(E)\simeq End(E')\simeq R$,   equipped with line bundles
$L\rightarrow E$, $L'\rightarrow E'$. Assume there are bijective maps $\sigma:E\rightarrow E'$,
$\tilde{\sigma}: L\rightarrow L'$ making commutative the following diagram
\begin{center}
 \begin{tikzpicture}[description/.style={fill=white,inner sep=4pt},normal line/.style={->,font=\footnotesize,shorten >=4pt,shorten <=4pt}]
 \matrix (m) [matrix of math nodes, row sep=3em, column sep=3.5em, text height=1.5ex, text depth=0.25ex]
 { L & L' \\ E & E' \\ };
 \path[normal line] (m-1-1) edge node[left] {p} (m-2-1);
 \path[normal line] (m-1-2) edge node[right] {q} (m-2-2);
 \path[normal line] (m-1-1) edge node[above] {$\tilde{\sigma}$} (m-1-2);
 \path[normal line] (m-2-1) edge node[above] {$\sigma$} (m-2-2);
 \end{tikzpicture}
 \end{center}
with $\sigma $ commuting with any $a\in R$.
Then for any $a\in R$ there exists a unique bijective $\tilde{\sigma}_a: a^*L\rightarrow a^*L'$ satisfying 
$q_a\circ \tilde{\sigma}_a=\sigma \circ p_a$ and $\tilde{a}\circ \tilde{\sigma}_a =\tilde{\sigma}\circ \tilde{a}
$ ($p_a: a^*L\to E$, $q_a: a^*L'\to E'$ denote the natural projections). Furthermore, we have
\begin{equation}\label{tildesigmacommutes}
\tilde{a}\circ \tilde{\sigma}_{ab}=\tilde{\sigma}_b\circ \tilde{a}: (ab)^*L\rightarrow b^*L', \hskip2mm  \forall a,b \in R.  
\end{equation}
\end{lemma}
\begin{proof}
Consider the maps $\sigma \circ p_a: a^*L \rightarrow E'$, $\tilde{\sigma}\circ \tilde{a}: a^*L \rightarrow L'.$
Since $\sigma $ commutes with $a$, we have
$$a\circ \sigma \circ p_a=\sigma \circ a \circ p_a= \sigma \circ p \circ \tilde {a}=q\circ \tilde{\sigma}\circ \tilde{a},$$
and the Universal Property implies there exists a unique $\tilde{\sigma}_a: a^*L\rightarrow a^*L'$ s.t.
$q_a\circ \tilde{\sigma}_a=\sigma \circ p_a$ and $\tilde{a}\circ \tilde{\sigma}_a =\tilde{\sigma}\circ \tilde{a}$.  
\par\noindent
In order to prove \ref{tildesigmacommutes} it suffices show that
\begin{enumerate}
\item $\tilde{b}\circ \tilde{a}\circ \tilde{\sigma}_{ab}=\tilde{b}\circ\tilde{\sigma}_b\circ \tilde{a}$,
\item $q_b\circ \tilde{a}\circ \tilde{\sigma}_{ab}=q_b\circ\tilde{\sigma}_b\circ \tilde{a}$.
\end{enumerate}
(1) from $\tilde{b}\circ \tilde{\sigma}_{b}=\tilde{\sigma}\circ \tilde{b}$ and $\tilde{ab}\circ \tilde{\sigma}_{ab}=\tilde{\sigma}\circ \tilde{ab}$  we have
$$\tilde{b}\circ \tilde{a}\circ \tilde{\sigma}_{ab}=  \tilde{\sigma}\circ \tilde{b}\circ \tilde{a}=\tilde{b}\circ \tilde{\sigma}_{b}\circ \tilde{a}.$$
(2) from  $q_b\circ \tilde{\sigma}_{b}=\sigma\circ p_b$ and $q_{ab}\circ \tilde{\sigma}_{ab}=\sigma\circ p_{ab}$ we have
$$q_b\circ\tilde{\sigma}_b\circ \tilde{a}=\sigma\circ p_b\circ \tilde{a}=\sigma\circ a\circ p_{ab}=a\circ \sigma\circ  p_{ab}=a\circ q_{ab}\circ \tilde{\sigma}_{ab}
=q_b\circ \tilde{a}\circ \tilde{\sigma}_{ab}.  
$$
\end{proof}

\begin{lemma}
\label{extensionphi} Keep the hypothesis of Lemma \ref{extensionsigma} and assume additionally that $L$ and $L'$ are obtained via pull-back of the hyperplane bundle
by means of embeddings $a: E\rightarrow \mathbb{P}(\C^n)$, $b: E'\rightarrow \mathbb{P}(\C^n)$ and that $\tilde{\sigma}$ is induced by a map
$\tilde{\sigma}: (\C^n)^* \rightarrow (\C^n)^*$,   obtained by acting on each coordinate with $\sigma \in \mathrm{Aut}(\mathbb{C})$.
Then we have an isomorphism
$$ \gc(L) \leftrightarrow \gc(L'), \hskip2mm \phi \leftrightarrow \phi^{\sigma}:=\tilde{\sigma} \circ \phi \circ \tilde{\sigma}^{-1}.   
$$
\end{lemma}
\begin{proof}
The proof is immediate since any $\phi \in \gc(L)$ correspond by a matrix $U\in GL(\C^n)$ via \textit{canonical representation} (\cite{CAV}, \S 6.4) and it is immediate to check that $\phi^{\sigma}$ is represented by $\sigma(U)$.  
\end{proof}

First of all we would like to study the behavior of the Adelic Heisenberg Group    under the action of field automorphisms of $\C$ fixing $E$.

Consider an elliptic curve $E$  with complex multiplication by $R$, embedded  in  $\mathbb{P}^2= \mathbb{P}(\C^3)$ by means of the Weierstrass model and assume it is   defined over a field $\mL $. Denote by $\sigma$ a field 
automorphisms   of $\C$ fixing $\K \cdot \mL$. By \cite{Sil} \S II, Theorem 4.1, $\sigma $ fix also
 $\mathbb{H}$, the  \textit{Hilbert field} of $\mathbb{K}$,  and may be interpreted as an element $\sigma \in Gal(\overline{\mathbb{H}},\mathbb{H})$.
Consider also a normalization $E\simeq \frac{\C}{\Lambda}$ (\cite{Sil}, \S II.1) with $\Lambda $ a fractional ideal of $\K$ and denote by
$L$  the line bundle providing the embedding $\frac{\C}{\Lambda}\to \mathbb{P}^2$. Observe that $L$ is symmetric so all the results of Section \ref{AHGsec}
can be applied.

\begin{theorem}
\label{sigmataucommute}
The field automorphism $\sigma $ acts as automorphism of $\hat{\mathcal{G}}(L)$  and we have: $\tau (\sigma (x))=\sigma (\tau (x))\in \hat{\mathcal{G}}(L)$, $\forall x\in V(E)$.
\end{theorem}
\begin{proof}
Since $\sigma $ fixes $\mL$, it provides a bijection of $E_{tor}$ which is induced (compare with Lemma \ref{extensionphi}), via embedding $E\subset \mathbb{P}^2$, by the bijective map
$\tilde{\sigma}: \C^3 \leftrightarrow \C^3$ (acting as $\sigma $ in any coordinate).
Such a map pull-back to $L$ via uniformization in such a way that we have a commutative diagram  
\begin{center}
 \begin{tikzpicture}[description/.style={fill=white,inner sep=4pt},normal line/.style={->,font=\footnotesize,shorten >=4pt,shorten <=4pt}]
 \matrix (m) [matrix of math nodes, row sep=3em, column sep=3.5em, text height=1.5ex, text depth=0.25ex]
 { L & L \\ E & E \\ };
 \path[normal line] (m-1-1) edge node[left] {p} (m-2-1);
 \path[normal line] (m-1-2) edge node[right] {p} (m-2-2);
 \path[normal line] (m-1-1) edge node[above] {$\tilde{\sigma}$} (m-1-2);
 \path[normal line] (m-2-1) edge node[above] {$\sigma$} (m-2-2);
 \end{tikzpicture}
 \end{center}
Moreover, by \cite{Sil}, Theorem 2.2 (b), the multiplication map by any $a\in R$ is still defined over $\K \cdot \mL$  so $a\circ \sigma=\sigma \circ a$, $\forall a\in R$ and we may apply Lemma \ref{tildesigmacommutes}.
\par\noindent
First of all we observe that $\sigma $ extend to an automorphism    
$$\sigma: V(E) \rightarrow  V(E)$$ 
because, if $x=(x_a)_{a\in R}$ belongs to $V(E)$ then also 
$\sigma(x):=(\sigma(x_a))_{a\in R}$ stays in $V(E)$ since 
$$a\sigma(x_{ab})=\sigma(ax_{ab})=\sigma(x_b).$$   Furthermore, we have
$$ay\in K(L)=E[d] \Leftrightarrow ady=0  \Leftrightarrow\sigma(ady)=ad \sigma(y)=0 \Leftrightarrow a\sigma(y)\in K(L),$$
so (compare with \ref{IeJ})
$$I_x=I_{\sigma(x)}, \hskip2mm J_x=J_{\sigma(x)}, \hskip2mm\forall x\in V(E).$$ 
If $\phi\in \gc(a^*L)$ define $\phi^{\sigma}:=\tilde{\sigma}_a \circ \phi \circ \tilde{\sigma}_a^{-1}\in  \gc(a^*L)$.  
Fix
$\alpha=(x, (\phi_a)_{a\in J_x})\in \hat{\mathcal{G}}(L)$ and set
$\alpha^{\sigma}=(\sigma(x), (\phi_a^{\sigma})_{a\in J_x})$.  Then \ref{tildesigmacommutes} implies
$$\tilde{a}\circ \phi^{\sigma}_{ab}= \tilde{a}\circ \tilde{\sigma}_{ab}  \circ\phi_{ab}\circ\tilde{\sigma}_{ab}^{-1}=
 \tilde{\sigma}_{b}\circ \tilde{a} \circ\phi_{ab}\circ\tilde{\sigma}_{ab}^{-1}=
 \tilde{\sigma}_{b} \circ\phi_{b}\circ \tilde{a}\circ\tilde{\sigma}_{ab}^{-1}=
 \tilde{\sigma}_{b} \circ\phi_{b}\circ\tilde{\sigma}_{b}^{-1}\circ \tilde{a}=\phi^{\sigma}_{b}\circ \tilde{a}$$
 so we have  $\alpha^{\sigma}\in   \hat{\mathcal{G}}(L)$   and we  have the desired  extension  of $\sigma$ to an automorphism of $\hat{\mathcal{G}}(L)$
$$\sigma: \hat{\mathcal{G}}(L)\rightarrow \hat{\mathcal{G}}(L).$$
\par\noindent
In order to conclude the proof we have 
$$\tau (\sigma (x))=\sigma(y)i(\sigma(y))^{-1},\hskip2mm \forall y\in V(E) \hskip1mm \mid \hskip1mm2g(y)=x,$$ 
since $2g(\sigma(y))=\sigma(2g(y))= \sigma (x)$ and $\tau(\sigma(x))  $ does not depend on $y'$ s.t.~$2g(y')=\sigma(x)$.
Finally, we have  
$$\tau (\sigma (x))=\sigma(y)i(\sigma(y))^{-1}=\sigma (yi(y)^{-1})=\sigma(\tau(x)),$$
because $\sigma: \hat{\mathcal{G}}(L)\rightarrow \hat{\mathcal{G}}(L)$ is a morphism commuting with $i$.    
\end{proof}
  
What we are going to do now is  to study the behavior of the Adelic Heisenberg Group   under the action of any field automorphisms of $\C$.
Consider again an elliptic curve $E$  with complex multiplication by $R$, embedded  in  $\mathbb{P}^2= \mathbb{P}(\C^3)$ by means of the Weierstrass model and assume it is   defined over a field $\mL $.
Recall the \textit{Main Theorem of Complex Multiplication} (\cite{Sil} II, Theorem 8.2, see also \cite{Sh}, Ch.~5 and \cite{EF}, Ch.~10):
\begin{theorem}
\label{MTCM}
Let $\sigma \in \mathrm{Aut}(\mathbb{C})$ fixing $\K$ and let $s$ be an id{\`e}le of $\K$ corresponding to $\sigma$ via Artin map. Fix a complex analytic isomorphism:
$$f:\frac{\mathbb{C}}{\Lambda}\rightarrow E(\mathbb{C}),$$
where $\Lambda$ is a fractional ideal. Then there exists a unique complex analytic isomorphism:
$$g:\frac{\mathbb{C}}{s^{-1}\Lambda}\rightarrow E^{\sigma}(\mathbb{C}),$$ so that the following diagram commutes:
 \begin{center}
 \begin{tikzpicture}[description/.style={fill=white,inner sep=4pt},normal line/.style={->,font=\footnotesize,shorten >=4pt,shorten <=4pt}]
 \matrix (m) [matrix of math nodes, row sep=2.5em, column sep=2.5em, text height=1.5ex, text depth=0.25ex]
 {  \frac{\mathbb{\K}}{\Lambda} & \frac{\mathbb{\K}}{s^{-1}\Lambda} \\
    E(\mathbb{C}) & E^{\sigma}(\mathbb{C}) \\  };
 \path[normal line] (m-1-1) edge node[above] {$s^{-1}$} (m-1-2);
 \path[normal line] (m-2-1) edge node[above] {$\sigma$} (m-2-2);
 \path[normal line] (m-1-1) edge node[left] {$f$} (m-2-1);
 \path[normal line] (m-1-2) edge node[right] {$g$} (m-2-2);
 \end{tikzpicture}
 \end{center}
\end{theorem}

The main purpose of the rest of this section is to lift such a commutative diagram to adelic Heisenberg groups.
To the ease notations we put $E'=E^{\sigma}.$ We may assume  $E(\mathbb{C})$ and $E(\mathbb{C})'$ both embedded $\mathbb{P}^2(\C)= \mathbb{P}(\C^3)$ by means of  Weierstrass models and we define $L:=f^*\oc_{\mathbb{P}^2}(1)$, $L':=g^*\oc_{\mathbb{P}^2}(1)$. Like in Lemma \ref{extensionphi},   $\sigma: E\rightarrow E'$ is induced by  
$\tilde{\sigma}: (\C^3)^*\rightarrow (\C^3)^*$ pulling back to a map $L\rightarrow L'$ and providing a commutative diagram

\begin{center}
\begin{tikzpicture}[description/.style={fill=white,inner sep=4pt},normal line/.style={->,font=\footnotesize,shorten >=4pt,shorten <=4pt}]
 \matrix (m) [matrix of math nodes, row sep=3em, column sep=3.5em, text height=1.5ex, text depth=0.25ex]
 { L & L' \\ E & E' \\ };
 \path[normal line] (m-1-1) edge node[left] {p} (m-2-1);
 \path[normal line] (m-1-2) edge node[right] {q} (m-2-2);
 \path[normal line] (m-1-1) edge node[above] {$\tilde{\sigma}$} (m-1-2);
 \path[normal line] (m-2-1) edge node[above] {$\sigma$} (m-2-2);
\end{tikzpicture}
\end{center}

\begin{lemma}\label{3points}~
\begin{enumerate}\itemsep=3pt
\item $a^{\sigma}=a$, $\forall a\in R$;
\item if $x, y\in  E$  are s.t. $ay=x$ then   $a\sigma(y)=\sigma(x)$;
\item $\sigma$ extends to an isomorphism 
$$\sigma: V(E)\rightarrow V(E), \hskip1mm x=(x_a)_{a\in R}\rightarrow \sigma(x):=(\sigma(x_a))_{a\in R}=s^{-1}x.
$$
\item $$I_x=I_{\sigma(x)}, \hskip2mm J_x=J_{\sigma(x)}, \hskip2mm\forall x\in V(E).$$  
\end{enumerate}
\end{lemma}
\begin{proof}
(1) First proof: since $E$ and $E'$ are normalized by the maps $f$ and $g$ of Theorem \ref{MTCM}, the statement follows directly  from \cite{Sil} II, Theorem 2.2.

Second proof: even more directly, Theorem \ref{MTCM} implies
$$ a^{\sigma}(x')=\sigma \circ a \circ \sigma ^{-1}(x')=\sigma(asx')= ax', \hskip2mm \forall x'\in E_{tor}.
$$ 
(2) $ ay=x$ $\Rightarrow$ $\sigma(ay)=\sigma(x)$ $\Rightarrow$  $a\sigma(y)=\sigma(x)$ thanks to (1).
\par\noindent
(3) It follows just combining (2) with Theorem \ref{MTCM} since $\sigma$ acts as the multiplication by $s^{-1}$ on each component $x_a$.
\par\noindent(4) We get
$$ax_1\in K(L)=E[d] \Leftrightarrow adx_1=0  \Leftrightarrow\sigma(adx_1)=ad \sigma(x_1)=0 \Leftrightarrow a\sigma(x_1)\in K(L'),$$
hence $I_x=I_{\sigma(x)}$, $\forall x\in V(E)$ and we are done by Lemma \ref{IeJ}.
\end{proof}

\begin{theorem}
\label{AMTCM}
The  isomorphism: $s^{-1}:V(E)\rightarrow V(E')$ lifts to an isomorphism of adelic Heisenberg groups 
$$ \sigma: \hat{\mathcal{G}}(L)\rightarrow\hat{\mathcal{G}}(L')
$$
commuting with the sections $\tau: V(E) \rightarrow \hat{\mathcal{G}}(L)$ and $\tau': V(E') \rightarrow \hat{\mathcal{G}}(L')$:
$$ \sigma(\tau(x))=\tau'(\sigma(x)),  \hskip5mm \forall x\in V(E).$$
\end{theorem}
\begin{proof}
Combining Lemmas \ref{extensionsigma}, \ref{extensionphi} and \ref{3points} the proof is very similar to that of Theorem  \ref{sigmataucommute}.   

By Lemma \ref{3points}, (1) we can apply Lemma \ref{extensionsigma} to $L$ and $L'$ so
 for any $a\in R$ there exists a unique bijective $\tilde{\sigma}_a: a^*L\rightarrow a^*L'$ satisfying 
$q_a\circ \tilde{\sigma}_a=\sigma \circ p_a$ and $\tilde{a}\circ \tilde{\sigma}_a =\tilde{\sigma}\circ \tilde{a}$   
and
\begin{equation}\label{AMTCMformula}
\tilde{a}\circ \tilde{\sigma}_{ab}=\tilde{\sigma}_b\circ \tilde{a}: (ab)^*L\rightarrow b^*L', \hskip2mm  \forall a,b \in R.  
\end{equation}

Furthermore, Lemma \ref{extensionphi}  implies that for any $a\in R$ we have an isomorphism
$$\gc(a^*L)\leftrightarrow\gc(a^*L) \hskip2mm \phi \leftrightarrow  \phi^{\sigma}:=\tilde{\sigma}_a \circ \phi \circ \tilde{\sigma}_a^{-1}.$$
Since $I_x=I_{\sigma(x)}$, $\forall x\in V(E)$ (compare with Lemma \ref{3points}, (4)),
for any
$\alpha=(x, (\phi_a)_{a\in J_x})\in \hat{\mathcal{G}}(L)$  we define
$\alpha^{\sigma}=(\sigma(x), (\phi_a^{\sigma})_{a\in J_x})$.  Then \ref{AMTCMformula} implies
$$\tilde{a}\circ \phi^{\sigma}_{ab}= \tilde{a}\circ \tilde{\sigma}_{ab}  \circ\phi_{ab}\circ\tilde{\sigma}_{ab}^{-1}=
 \tilde{\sigma}_{b}\circ \tilde{a} \circ\phi_{ab}\circ\tilde{\sigma}_{ab}^{-1}=
 \tilde{\sigma}_{b} \circ\phi_{b}\circ \tilde{a}\circ\tilde{\sigma}_{ab}^{-1}=
 \tilde{\sigma}_{b} \circ\phi_{b}\circ\tilde{\sigma}_{b}^{-1}\circ \tilde{a}=\phi^{\sigma}_{b}\circ \tilde{a}$$
 so we have  $\alpha^{\sigma}\in   \hat{\mathcal{G}}(L)$   and we  have the desired  extension  of $\sigma$ to an automorphism of $\hat{\mathcal{G}}(L)$
$$\sigma: \hat{\mathcal{G}}(L)\rightarrow \hat{\mathcal{G}}(L).$$
\par\noindent
In order to conclude the proof we have 
$$\tau' (\sigma (x))=\sigma(y)i(\sigma(y))^{-1},\hskip2mm \forall y\in V(E) \hskip1mm \mid \hskip1mm2g(y)=x,$$ 
since $2g(\sigma(y))=\sigma(2g(y))= \sigma (x)$ and $\tau(\sigma(x))  $ does not depend on $y'$ s.t.~$2g(y')=\sigma(x)$.
Finally, we have  
$$\tau' (\sigma (x))=\sigma(y)i(\sigma(y))^{-1}=\sigma (yi(y)^{-1})=\sigma(\tau(x)),$$
because $\sigma: \hat{\mathcal{G}}(L)\rightarrow \hat{\mathcal{G}}(L)$ is a morphism commuting with $i$. 
\end{proof}

\section{Adelic Thetas}

What we are going to do in this section is to show that \textit{canonical representations}   $\tilde{\rho}_a$ defined in \ref{bhat}, fit for different $a\in R$ to give a representation $U$ (Proposition \ref{morphism})   of the adelic Heisenberg group $ \hat{\mathcal{G}}(L)$ into the direct limit (Definition \ref{U})
$$\hat{H}^0(L)\simeq \varinjlim_{a\in R}H^0(a^*L).$$
This allows us to define \textit{adelic theta functions} $\theta_s$, for any $s\in \hat{H}^0(L)$ (Definition \ref{thetafctn}),  on $V(E)$ by means of the lifting
$\tau: V(E) \to \hat{\mathcal{G}}(L)$ defined in \ref{tau} (in all this section the line bundle $L$ is assumed to be symmetric).
We obtain in such a way (Proposition \ref{propertiestheta}) a vector space of functions over which $ \hat{\mathcal{G}}(L)$ is represented by means of \textit{translations and characters} (as it usually happens in canonical representations). 

Combining all that with the properties of the adelic action on $V(E)$ stemming from the main theorem of 
complex multiplication for elliptic curves (proved in \S 5) we find  a nice intertwining between theta functions and $\C$-automorphisms (Theorems \ref{main} and \ref{main'}).

Last but not least, composing theta functions with the embeddings defined in Notations \ref{mapslattices}, we are going to define theta functions exhibiting a nice behavior
 under $\C$-automorphisms 
on commensurability classes of arithmetic $1$-dimensional $\mathbb{K}$-lattices (Theorem \ref{last})
and on the groupoid of commensurability modulo dilations (Notations \ref{lastnotations}, Theorem \ref{last'}).

\begin{definition}\label{bhat}
By an abuse of notations, for any $\psi \in \mathcal{G}(a^*L)_x$ we still denote by $\psi$ its image via \textit{canonical representation} (\cite{CAV}, 6.4):
$$\tilde{\rho}_a: \mathcal{G}(a^*L)\rightarrow GL(H^0(a^*L)),\hskip2mm \psi=\tilde{\rho}_a(\psi): s\rightarrow \psi\circ s \circ t_{-x}.$$

If $s\in H^0(a^*L)$ and $b\in R$ then $b\circ id=b=id\circ b=p_a \circ s \circ b$ so, by  universal property, there exists   $\hat{b}(s)\in H^0((ab)^*L)$ s.t.
$\tilde{b} ( \hat{b}(s))=s\circ b$. We find a linear  map
$$\hat{b}: H^0(a^*L)\hookrightarrow H^0((ab)^*L).$$   
\end{definition}
\bigskip

\begin{lemma}
\label{commuting}
Consider $\psi \in \mathcal{G}((ab)^*L)$ and $\phi \in \mathcal{G}(a^*L)$ s.t.~$\tilde{b}\circ \psi=\phi\circ \tilde{b}$.
Then we have a commutative square:
 \begin{center}
 \begin{tikzpicture}[description/.style={fill=white,inner sep=4pt},normal line/.style={->,font=\footnotesize,shorten >=4pt,shorten <=4pt}]
 \matrix (m) [matrix of math nodes, row sep=2.5em, column sep=2.5em, text height=1.5ex, text depth=0.25ex]
 {  H^0(a^*L) & H^0((ab)^*L) \\
    H^0(a^*L) & H^0((ab)^*L) \\  };
 \path[normal line] (m-1-1) edge node[above] {$\hat{b}$} (m-1-2);
 \path[normal line] (m-2-1) edge node[above] {$\hat{b}$} (m-2-2);
 \path[normal line] (m-1-1) edge node[left] {$\phi$} (m-2-1);
 \path[normal line] (m-1-2) edge node[right] {$\psi$} (m-2-2);
 \end{tikzpicture}
 \end{center}
\end{lemma}
\begin{proof}
For any $s \in H^0(a^*L)$, $\hat{b}(s)\in H^0((ab)^*L)$ is characterized by $\tilde{b}\circ \hat{b} (s)=s\circ b$ so we are left to prove that
$\tilde{b}\circ \psi (\hat{b}(s))=\phi(s) \circ b$, $\forall s\in   H^0(a^*L)$. We have
\begin{multline*}
\tilde{b}\circ \psi (\hat{b}(s))=\tilde{b}\circ \psi \circ \hat{b}(s)\circ t_{-bx}=\phi\circ \tilde{b}\circ \hat{b}(s)\circ t_{-x} \\
=\phi\circ s\circ b\circ t_{-x}=\phi\circ s\circ  t_{-bx}\circ b= \phi(s) \circ b. \qedhere
\end{multline*}
\end{proof}

\begin{definition}\label{U}~
\begin{enumerate}\itemsep=3pt
\item Set 
$$\hat{H}^0(L)\simeq \varinjlim_{a\in R}H^0(a^*L), \hskip2mm \iota_a:H^0(a^*L)\rightarrow \hat{H}^0(L),$$ with $\iota_a$ denoting the canonical inclusion.
For any $\alpha =(x,  (\phi_a)_{a\in J_{x}})\in\hat{\mathcal{G}}(L)$, define
$$U_{\alpha}: \hat{H}^0(L)\rightarrow \hat{H}^0(L), \hskip5mm U_{\alpha}\mid _{H^0(a^*L)}=\tilde{\rho}_a(\phi _a), \hskip1mm \forall a\in J_x.$$
Such a   $U_{\alpha}$ is well defined by Lemma \ref{commuting}.
\item For any $a\in R$ we denote by $GL_a$ the group
$$ GL_a:=\{ (\phi_r)_{r\in (a)} \mid \hskip1mm \phi_b \in GL(H^0(b^*L)), \hskip1mm \phi_{ab}\circ \hat{c}=\hat{c} \circ \phi_{abc} \hskip1mm \forall b, c\in R\}.
$$
\item If $a \mid b$ then $(b)\subset (a)$ and there an obvious injective group morphism $GL_a\hookrightarrow GL_b.$ So   can define the limit   group
$$GL(\hat{H}^0(L)):=\varinjlim_{a\in R} GL_a.$$
\end{enumerate}
\end{definition}

\begin{remark}
Observe that Lemma \ref{commuting} implies that $U_{\alpha}$ defined in  \ref{U}, (1) belongs to $GL(\hat{H}^0(L)).$
The following Proposition shows that the correspondence  $\alpha \rightarrow U_{\alpha}$ is indeed a representation of $\hat{\mathcal{G}}(L)$ in 
$GL(\hat{H}^0(L)).$
\end{remark}
   .
\begin{proposition}
\label{morphism} Let $L$ be very ample, choose a section $s\in H^0(L)$ and assume everything defined over some field $\mL$.
The map defined in  \ref{U}, (1)
$$U: \hat{\mathcal{G}}(L)\rightarrow  GL(\hat{H}^0(L)), \hskip3mm \alpha \rightarrow U_{\alpha}$$ is a group morphism.  Furthermore, we have:
$$U_{\tau(x)}\circ U_{\tau(y)}(s)=\tilde{e}(\tfrac{x}{2},y)U_{\tau (x+y)}(s).$$  
\end{proposition}
\begin{proof}
Fix $\alpha=(x,(\phi_a)_{a\in J_x})\in \hat{\mathcal{G}}(L)$ and $\beta=(x,(\phi_a)_{a\in J_y})\in \hat{\mathcal{G}}(L)$. 
If $a\in J_x \cap J_y$, then $\alpha $ and $\beta $ belong to $\hat{\mathcal{G}}(L)\mid_{\vc_a}$ (compare with Definition \ref{AHG}, (3)).
What we are going to prove is that 
$$U:\hat{\mathcal{G}}(L)\mid_{\vc_a} \rightarrow GL(\hat{H}^0(L))$$ is a group morphism.   
Theorem \ref{exactseq}, (2) implies
$$\hat{\mathcal{G}}(L)\mid_{\vc_a}\simeq \nu_a^*\gc(a^*L).$$
Moreover, $U:\nu_a^*\gc(a^*L) \rightarrow GL(\hat{H}^0(L))$   obviously factorizes through $GL_a$ 
$$U: \nu_a^*\gc(a^*L)\rightarrow GL_a \subset GL(\hat{H}^0(L))  
$$
and we have a commutative diagram   (compare with Definition \ref{bhat}):
 \begin{center}
 \begin{tikzpicture}[description/.style={fill=white,inner sep=4pt},normal line/.style={->,font=\footnotesize,shorten >=4pt,shorten <=4pt}]
 \matrix (m) [matrix of math nodes, row sep=2.5em, column sep=2.5em, text height=1.5ex, text depth=0.25ex]
 {  \nu_a^*\gc(a^*L) & GL_a \\
    \mathcal{G}(a^*L) & GL(H^0(a^*L)) \\  };
 \path[normal line] (m-1-1) edge node[above] {$U$} (m-1-2);
 \path[normal line] (m-2-1) edge node[above] {$\rho_a$} (m-2-2);
 \path[normal line] (m-1-1) edge node[left] {} (m-2-1);
 \path[normal line] (m-1-2) edge (m-2-2);
 \end{tikzpicture}
 \end{center}
i.e. the map $U\mid_{\hat{\mathcal{G}}(L)\mid_{\vc_a}}$ factorizes  through $\rho_a$ and must be a morphism since $\rho_a$ it is. Finally, 
$$U_{\tau(x)}\circ U_{\tau(y)}(s)=\tilde{e}(\tfrac{x}{2},y)U_{\tau (x+y)}(s)$$
follows from Proposition \ref{notmorphism}.
\end{proof}

Recall \cite{Theta3}, Definition 5.5: 

\begin{definition}
\label{thetafctn} Fix $x\in V(E)$, $s\in \iota_b(H^0(b^*L)) \subset \hat{H}^0(L)$ and assume $l\in L(0)^*$, also defined on $\mathbb{L}$.  
We define the \textit{adelic theta function associated to}  $s$:
$$\theta _s: V(E) \rightarrow \overline{\mathbb{L}},$$
in such a way that
$$\theta _s(x)=l\big(\phi_{ab}^{-1}\big(\hat{a}\big( s\big)\big(x_{ab})\big), \hskip1mm \forall a \mid ab\in  J_x,$$
if $\tau(x)=\big(x,  (\phi_c)_{c\in J_{x}}\big)$.
\end{definition}
Such a theta function is well defined in view of the following:

\begin{lemma}
If both $ ab$ and $ cb$ belong to $  J_x$ then 
$$\phi_{ab}^{-1}(\hat{a}(s)(x_{ab}))=\phi_{cb}^{-1}(\hat{c}(s)(x_{cb}))=\phi_{acb}^{-1}(\hat{ac}(s)(x_{acb})).$$
\end{lemma}
\begin{proof}
We prove 
$$\phi_{ab}^{-1}(\hat{a}(s)(x_{ab}))=\phi_{acb}^{-1}(\hat{ac}(s)(x_{acb})).
$$ 
Since $\phi_{acb}^{-1}(\hat{ac}(s)(x_{acb}))\in (acb)^*L\mid_{0}$ and any $\tilde{c}$ acts as the identity on the zero-fiber,
we have
\begin{multline*}
\phi_{acb}^{-1}(\hat{ac}(s)(x_{acb}))=\tilde{c}(\phi_{acb}^{-1}(\hat{ac}(s)(x_{acb}))) \\
 =\phi_{ab}^{-1}(\tilde{c}(\hat{c}\circ\hat{a} (s)(x_{acb}))) 
=\phi_{ab}^{-1}(\hat{a}(s)(cx_{acb}))=\phi_{ab}^{-1}(\hat{a}(s)(x_{ab})),
\end{multline*}
because $\tilde{c}(\hat{c}(s))=s\circ c $, $\forall s\in H^0((ab)^*L)$
(compare with Definition \ref{bhat}).
\end{proof}

\smallskip

We recall some properties of adelic theta functions (see \cite{Theta3}, Chap.~5):
\begin{proposition}\label{propertiestheta}~\vspace{3pt}
\begin{enumerate}\itemsep=5pt
\item \qquad $ \theta _s(x)=l(U_{\tau(-x)}s) $
\item \qquad $ \theta_{U_{\tau(y)}s}(x)=\tilde{e}(y,\frac{x}{2})\theta_s(x-y) $
\end{enumerate}
\end{proposition}
\begin{proof}
(1) Assume that $\tau(x)=(x,  (\phi_a)_{a\in J_{x}})$. By Proposition \ref{morphism}, we have
 $$U_{\tau(x)}\circ U_{\tau(-x)}=\tilde{e}(\tfrac{x}{2},-x)U_{\tau (0)}=id$$
so $U_{\tau(x)}^{-1}= U_{\tau(-x)}$ with
$\tau(-x)=(-x,  (\phi_a^{-1})_{a\in J_{x}})$. In order to ease the notations, if $s\in \iota_b(H^0(b^*L)) \subset \hat{H}^0(L)$ we set
$s_{ab}=\hat{a}(s)$, if $ab\in  J_x$, so we have:
$$U_{\tau(-x)}s=\phi_{ab}^{-1}s_{ab}(x_{ab}),\hskip2mm \mathrm{and} \hskip2mm \theta _s(x)=l(\phi_{ab}^{-1}s_{ab}(x_{ab}))=l(U_{\tau(-x)}s).$$
(2) It follows combining (1) with Proposition \ref{morphism}:
$$\theta_{U_{\tau(y)}s}(x)=l(U_{\tau(-x)\circ\tau(y)}s)=
\tilde{e}(y,\tfrac{x}{2})l(U_{\tau(y-x)}s)=\tilde{e}(y,\tfrac{x}{2})\theta_s(x-y). \vspace{-20pt}
$$
\end{proof}

\begin{theorem}
\label{main}~
\begin{enumerate}\itemsep=3pt
\item
As in Theorem \ref{sigmataucommute}, consider an elliptic curve $E$  with complex multiplication by $R$, embedded  in  $\mathbb{P}^2= \mathbb{P}(\C^3)$ by means of the Weierstrass model and assume it is   defined over a field $\mL $. Denote by $\sigma$ a field 
automorphisms   of $\C$ fixing $\K \cdot \mL$. Consider also a normalization $E\simeq \frac{\C}{\Lambda}$ (\cite{Sil}, \S II.1) with $\Lambda $ a fractional ideal of $\K$ and denote by
$L$  the line bundle providing the embedding $\frac{\C}{\Lambda}\to \mathbb{P}^2$. Fix a section $s\in H^0(L)$ corresponding to a line of $\mathbb{P}^2$ also defined on 
$\K \cdot \mL$.
Then we have:
$$\sigma(\theta_s(x))=\theta_{s}(\sigma(x))=\theta_{s}(l\cdot x),  $$
where $l\in \As $ is the  id{\`e}le of $\K$ corresponding to 
$\sigma: E_{tor}\to E_{tor}$ (according to \S $2$).
\item
Keep notations as in Theorems \ref{MTCM} and \ref{AMTCM}, define $l\in \As $ as the inverse of the  id{\`e}le  corresponding to $\sigma$ via Artin map,
 consider   $E(\mathbb{C})$ and $E(\mathbb{C})'$ both embedded $\mathbb{P}^2(\C)$ by means of  Weierstrass models and  define $L:=f^*\oc_{\mathbb{P}^2}(1)$, $L':=g^*\oc_{\mathbb{P}^2}(1)$.
Fix sections 
$s\in H^0(E,L)$ and $s'\in H^0(E',L')$ corresponding to the same line in $\mathbb{P}^2(\mathbb{K})$.
Then we have:
$$\sigma(\theta_s(x))=\theta_{s'}(\sigma(x))=\theta_{s'}(l\cdot x).  $$
\end{enumerate}
\end{theorem}
\begin{proof} We are going to prove (2) as the proof of (1) runs the same way.
By Theorem \ref{AMTCM}, $\tau' (\sigma (x))=\sigma (\tau (x))\in \hat{\mathcal{G}}(L')$, $\forall x\in V(E)$ so 
$\tau' (\sigma (x))=(\sigma(x),  (\phi^{\sigma}_a)_{a\in J_{x}})$. 
Then we have:
$$\theta_{s'}(\sigma(x))=l((\phi^{\sigma}_a)^{-1}(s'_a(\sigma(x_a))))=l((\phi^{\sigma}_a)^{-1}(\sigma(s_a(x_a)))),  $$ by definition of $s$ and $s'$,
$$l((\phi^{\sigma}_a)^{-1}(\sigma(s(x_a))))=l(\sigma((\phi_a)^{-1}(s(x_a))))$$
 by definition of 
 $\sigma: \gc (a^*L)\rightarrow \gc (a^*L'),$ and finally
 $$l(\sigma((\phi_a)^{-1}(s(x_a))))=\sigma(l((\phi_a)^{-1}(s(x_a))))=\sigma (\theta_s(x))$$ since $l$ is  defined over $\mathbb{K}.$
\end{proof}
We have furthermore the important Corollary (see \cite{Theta3}, Proposition 5.6):

\begin{definition}(Compare with \cite{Sil}, Theorem 8.2)
\label{chi}
 Fix an automorphism of the complex numbers $\sigma$, and assume $\sigma\mid_{\mathbb{K}^{ab}}=[t,\mathbb{K}]$, $\sigma\mid_{\mathbb{Q}^{ab}}=[r,\mathbb{Q}]$ via Artin maps:
 $$[\,\cdot\,, \mathbb{K}]:\mathbb{A}^*_{\mathbb{K}}\rightarrow Gal(\mathbb{K}^{ab},\mathbb{K}), \hskip4mm [\,\cdot\,, \mathbb{Q}]:\mathbb{A}^*_{\mathbb{Q}}\rightarrow Gal(\mathbb{Q}^{ab},\mathbb{Q}).$$
We define: 
$$\chi_{\sigma}= V(E)\times V(E)\rightarrow  \overline{\mathbb{Q}}^1 ,\hskip5mm \chi_{\sigma}(x,y)=\frac{r^{-1}\tilde{e}(x,y)}{\tilde{e}(s^{-1}x,s^{-1}y)}.$$
\end{definition}

We state our main result concerning the behaviour of adelic theta functions under automorphisms:
\begin{theorem}
\label{main'}~
\begin{enumerate}\itemsep=3pt
\item With notations as in Theorem \ref{main} (1),  let $t$ and $r$ be  ideles of $\mathbb{K}$ and 
$\mathbb{Q}$ corresponding to $\sigma$ via Artin maps. Then  we have:
$$\sigma (\theta _{U_{\tau(y)}s}(x))=\chi_{\sigma}(y,\tfrac{x}{2})\theta_{U_{\tau(t^{-1}y)}s}(t^{-1}x).$$
\item
With notations as in Theorem \ref{main} (2),  let $t$ and $r$ be  ideles of $\mathbb{K}$ and 
$\mathbb{Q}$ corresponding to $\sigma$ via Artin maps. Then  we have:
$$\sigma (\theta _{U_{\tau(y)}s}(x))=\chi_{\sigma}(y,\tfrac{x}{2})\theta_{U_{\tau(t^{-1}y)}s'}(t^{-1}x).$$
\end{enumerate}
\end{theorem}
\begin{proof} We are going to prove (2) as the proof of (1) runs the same way.
By Proposition \ref{propertiestheta}, $\theta _{U_{\tau(y)}s}(x)=\tilde{e}(y,\tfrac{x}{2})\theta_s(x-y)$, so Theorem \ref{main} implies:
\begin{multline*}
\sigma (\theta _{U_{\tau(y)}s}(x))=\sigma(\tilde{e}(y,\tfrac{x}{2})\theta_s(x-y)) \\
=r^{-1}\tilde{e}(y,\tfrac{x}{2})\theta_{s'}(\sigma(x-y))=
r^{-1}\tilde{e}(y,\tfrac{x}{2})\theta_{s'}(t^{-1}(x-y)),
\end{multline*}
and we conclude by applying Proposition \ref{propertiestheta} once again.
\end{proof}

Finally, we may  apply our results to the set of arithmetic  1-dimensional $\K$  lattices modulo commensurability and to the groupoid of commensurability modulo dilations. 

\begin{theorem}\label{last}
Keep notations as  in Theorem \ref{main} (1) 
and let   $l\in \As $ be an id{\`e}le of $\K$ corresponding to 
$\sigma: E_{tor}\to E_{tor}$ (according to \S $2$). 
Recall the map 
$$\rho_{\Lambda}:\fc \rightarrow V(E_{\Lambda}), \hskip4mm \rho_{\Lambda}\mid_{T(E_{\Lambda'})^*}:=V(\alpha_{\Lambda, \Lambda'})^{-1},$$
defined in \ref{mapslattices} and acting on the set of arithmetic  1-dimensional $\K$  lattices modulo commensurability $\fc$. Set
\begin{equation}\label{lasteq}
\tilde{\theta}_s:=  \theta_s \circ \rho_{\Lambda}: \fc \longrightarrow \overline{\mathbb{L}}.
\end{equation}

Then we have
$$\sigma(\tilde{\theta}_s(x))=\tilde{\theta}_{s}(l\cdot x).  $$
\end{theorem}
\begin{proof} Recall that $\fc = \bigcup_{\Lambda'}T(E_{\Lambda'})^*$ by Corollary \ref{invertiblelattices}.
By \ref{mapslattices} (2) and \ref{lasteq}, 
then
$$\tilde{\theta}\mid_{T(E_{\Lambda'})^*}= \theta_s \circ V(\alpha_{\Lambda, \Lambda'})^{-1}.$$
Then it follows from Theorem \ref{main}
that 
\begin{multline*}
\sigma(\tilde{\theta}_s(x))=\sigma(\theta_s(V(\alpha_{\Lambda, \Lambda'})^{-1}(x)))=
\theta_s(l\cdot V(\alpha_{\Lambda, \Lambda'})^{-1}(x)))= \\
\theta_s(V(\alpha_{\Lambda, \Lambda'})^{-1}(l\cdot x)))=\tilde{\theta}_{s}(l\cdot x),
\hskip2mm \forall x\in T(E_{\Lambda'})^*
\end{multline*}
(the third equality depends on the fact that $V(\alpha_{\Lambda, \Lambda'})^{-1}$ is just the multiplication by an id{\`e}le once one has fixed an element in $T(E)^*$).
We are done because of $\fc=\bigcup_{\Lambda'}T(E_{\Lambda'})^*$.
\end{proof}

\begin{notations}\label{lastnotations}
Like in Theorem \ref{groupoidthm} fix a set of representatives $\Lambda _i$ of $Cl(R)$, 
 $1\leq i \leq \sharp Cl(R)$,  and define $E_i:=\frac{\C}{\Lambda_i}$. Let $\sigma \in \mathrm{Aut}(\mathbb{C})$ fixing $\K$ and let $s$ be an id{\`e}le of $\K$ corresponding to $\sigma$ via Artin map. Set moreover  $E_i':=\sigma(E_i)$ like in Theorems \ref{MTCM} and \ref{AMTCM},   consider 
$E(\mathbb{C})$ and $E(\mathbb{C})'$ both embedded $\mathbb{P}^2(\C)$ by means of  Weierstrass models and  define $L_i:=f_i^*\oc_{\mathbb{P}^2}(1)$, $L_i':=g_i^*\oc_{\mathbb{P}^2}(1)$.
Fix sections 
$s_i\in H^0(E_i,L_i)$ and $s'_i\in H^0(E'_i,L_i')$ corresponding to the same line in $\mathbb{P}^2(\mathbb{K})$.
Recall (compare with \ref{mapslattices}) that we defined
 a map from  the groupoid of commensurability modulo dilations $\Sc $   to $\bigcup_{\Lambda_i }  V(E_{\Lambda_i})$:
 
$$\xi: \Sc = \bigcup_{(\Lambda_i,\Lambda)\in\ec  }\fc _{\Lambda_i, \Lambda}\longrightarrow  \bigcup_{\Lambda_i }  V(E_{\Lambda_i}), \hskip3mm
\xi\mid_{\fc _{\Lambda_i, \Lambda}}: (x,y)\rightarrow   V(\alpha_{\Lambda_i, \Lambda})^{-1}(y)\cdot x. $$
\end{notations}

The following Theorem can be proved just as Theorem \ref{last}

\begin{theorem}\label{last'}
Keeping notations as above, define $$\Theta, \Theta': \Sc \to \overline{\mathbb{L}}$$ by
$$\Theta\mid_{\fc _{\Lambda_i, \Lambda}}:=\theta_{s_i} \circ \xi, \hskip2mm \Theta'\mid_{\fc _{\Lambda_i, \Lambda}}:=\theta_{s_i'} \circ \xi.
$$
Then we have
$$\sigma(\Theta(x,y))=\Theta'(s^{-1}x,y).  $$
\end{theorem}


\begin{thebibliography}{10}

\bibitem{CAV} Ch. Birkenhake, H. Lange: Complex Abelian Varieties. 2nd augmented ed.. Grundlehren der Mathematischen Wissenschaften 302. Berlin: Springer (2004).

\bibitem{BC} J.B. Bost, A. Connes: Hecke algebras, type III factors and phase transitions with spontaneous symmetry breaking in number theory. 
Sel. Math., New Ser. 1, No. 3, 411-457 (1995).

\bibitem{LocF} J.W.S. Cassels: Local fields. London Mathematical Society Student Texts, 3. Cambridge etc.: Cambridge University Press (1986).

\bibitem{Ch} P.B. Cohen: A $C^*$-dynamical system with Dedekind zeta partition function and spontaneous symmetry breaking. 
J. Théor. Nombres Bordx. 11, No.1, 15-30 (1999).

\bibitem{CMBook} A. Connes, M. Marcolli: 
Noncommutative Geometry, Quantum Fields and Motives. 
Colloquium Publications, 55. American Mathematical Society (2005).

\bibitem{KMS} A. Connes, M. Marcolli, N. Ramachandran: 
KMS states and complex multiplication. 
Sel. Math., New Ser. 11, No. 3-4, 325-347 (2005).

\bibitem{KMSII} A. Connes, M. Marcolli, N. Ramachandran: 
KMS states and complex multiplication. II. 
Bratteli, Ola (ed.) et al., Operator algebras. The Abel symposium 2004. Proceedings of the first Abel symposium, Oslo, Norway, September 3--5, 2004. Berlin: Springer. Abel Symposia 1, 15-59 (2006).

\bibitem{HL} D. Harari, E. Leichtnam: Extension du ph\'enom\`ene de brisure spontan\'ee de sym\'etrie de Bost-Connes au cas des corps globaux quelconques.
Sel. Math., New Ser. 3, No.2, 205-243 (1997).

\bibitem{Hus}  D. Husem\"{o}ller: Elliptic Curves. Graduate Texts in Mathematics, 111. New York etc.: Springer-Verlag (2004). 

\bibitem{LF} M. Laca, M. van Frankenhuijsen: Phase transitions on Hecke $C^*$-algebras and class-field theory over $\mathbb{Q}$. 
J. Reine Angew. Math. 595, 25-53 (2006).

\bibitem{ANT}  S. Lang: Algebraic Number Theory. Graduate Texts in Mathematics. 110. New York: Springer-Verlag (1994).

\bibitem{EF}  S. Lang: Elliptic Functions. Graduate Texts in Mathematics, 112. New York etc.: Springer-Verlag (1987). 

\bibitem{Mar}  D. A. Marcus: Number Fields. Univeritext. New York etc.: Springer-Verlag (Second Printing, 1987). 

\bibitem{Mum} D. Mumford: On the Equations Defining Abelian Varieties, I-III. Invent. Math. 1, 287-354 (1966); ibid. 3, 75-135, 215-244 (1967).

\bibitem{Theta3} D. Mumford: Tata lectures on theta III. In collaboration with Madhav Nori and Peter Norman. Reprint of the 1991 edition. 
Modern Birkhäuser Classics. Basel: Birkhäuser (2007).

\bibitem{ProGr} L. Ribes, P. Zalesskii: Profinite groups. 2nd ed. 
Ergebnisse der Mathematik und ihrer Grenzgebiete. 3. Folge 40. Berlin: Springer (2010).

\bibitem{Sil0} J. H. Silverman:  The Arithmetic of Elliptic curves. 2nd ed.
Graduate Texts in Mathematics. 106. New York, NY: Springer-Verlag (2009).

\bibitem{Sil} J. H. Silverman: Advanced Topics in the Arithmetic of Elliptic curves.
Graduate Texts in Mathematics. 151. New York, NY: Springer-Verlag (1994). 

\bibitem{Sh} G. Shimura: Introduction to Arithmetic theory of Automorphic Functions. Repr. of the 1971 orig.
Publications of the Mathematical Society of Japan. 11.  Princeton University Press (1994).

\end{thebibliography}
\end{document}